\documentclass[12pt]{amsart}
\usepackage{txfonts}      
\usepackage{amssymb}
\usepackage{eucal}
\usepackage{amsmath}
\usepackage{amscd}
\usepackage[dvips]{color}
\usepackage{multicol}
\usepackage[all]{xy}           
\usepackage{graphicx}
\usepackage{color}
\usepackage{colordvi}
\usepackage{xspace}
\usepackage{tikz}
\usepackage{makecell}

\usepackage{ifpdf}
\ifpdf
\usepackage[colorlinks,final,backref=page,hyperindex]{hyperref}
\else
\usepackage[colorlinks,final,backref=page,hyperindex,hypertex]{hyperref}
\fi

\usepackage[active]{srcltx} 

\usepackage{bbding}
\usepackage[shortlabels]{enumitem}
\usepackage{tikz-cd}

\usepackage{difftrees}

\makeatletter

\newtheorem{thm}{Theorem}[section]
\newtheorem{lem}[thm]{Lemma}
\newtheorem{cor}[thm]{Corollary}
\newtheorem{pro}[thm]{Proposition}
\theoremstyle{definition}   
\newtheorem{defi}[thm]{Definition}

\newtheorem{ex}[thm]{Example}
\newtheorem{rmk}[thm]{Remark}

\newcommand{\nc}{\newcommand}
\newcommand{\delete}[1]{}

\nc{\mlabel}[1]{\label{#1}}  
\nc{\mcite}[1]{\cite{#1}}  
\nc{\mref}[1]{\ref{#1}}  
\nc{\meqref}[1]{\eqref{#1}}  
\nc{\mbibitem}[1]{\bibitem{#1}} 

\delete{
\nc{\mlabel}[1]{\label{#1}{\hfill \hspace{1cm}{\tt{{\ }\hfill(#1)}}}}
\nc{\mcite}[1]{\cite{#1}{{\tt{{\ }(#1)}}}}  
\nc{\mref}[1]{\ref{#1}{{\tt{{\ }(#1)}}}}  
\nc{\meqref}[1]{\eqref{#1}{{\tt{{\ }(#1)}}}}  
\nc{\mbibitem}[1]{\bibitem[\tt #1]{#1}} 
}

\newcommand {\emptycomment}[1]{}

\newcommand{\bk}{{\mathbf{k}}}
\nc{\vep}{\varepsilon}

\nc{\oprn}{\theta}
\nc{\Oprn}{\Theta}

\nc{\tforall}{\ \ \text{for all }}

\nc{\calo}{\mathcal{O}}
\nc{\oop}{$\mathcal{O}$-operator\xspace}
\nc{\oops}{$\mathcal{O}$-operators\xspace}
\nc{\mrho}{{\bm{\varrho}}}
\nc{\emk}{\mathbf{K}}
\nc{\invlim}{\displaystyle{\lim_{\longleftarrow}}\,}
\nc{\ot}{\otimes}
\newcommand{\co}{\mathsf{cosh}}

\newcommand{\lon }{\,\rightarrow\,}
\newcommand{\be }{\begin{equation}}
\newcommand{\ee }{\end{equation}}

\newcommand{\g}{\mathfrak g}
\newcommand{\h}{\mathfrak h}

\newcommand{\huaB}{\mathcal{B}}

\newcommand{\huaF}{\mathcal{F}}

\newcommand{\huaU}{\mathcal{U}}

\newcommand{\huaY}{\mathcal{Y}}
\newcommand{\huaQ}{\mathcal{Q}}
\newcommand{\huaP}{\mathcal{P}}
\nc{\calp}{\mathcal{P}}

\newcommand{\huaH}{\mathcal{H}}

\newcommand{\huaO}{{\mathcal{PT}}}
\newcommand{\huaT}{\mathcal{T}}

\newcommand{\Id}{{\rm{Id}}}

\newcommand{\br}[1]{   [ \cdot,    \cdot  ]   }

\newcommand{\dM}{\mathrm{d}}

\newcommand{\Hom}{\mathrm{Hom}}
\newcommand{\Der}{\mathrm{Der}}

\newcommand{\PG}{\mathsf{PG}}
\newcommand{\BCK}{\mathrm{BCK}}
\newcommand{\LB}{\mathrm{LB}}

\newcommand{\FF}{\mathsf{F}}
\newcommand{\GG}{\mathsf{G}}
\newcommand{\alg}{\mathsf{alg}}
\newcommand{\BCH}{\mathsf{BCH}}

\newcommand{\BS}{\mathsf{NBS}}
\nc{\NBS}{\mathsf{NBS}}
\newcommand{\BG}{\mathsf{BG}}

\newcommand{\SLB}{\mathsf{SLB}}

\newcommand{\Ad}{\mathrm{Ad}}
\newcommand{\Aut}{\mathrm{Aut}}

\newcommand{\gl}{\mathfrak {gl}}

\newcommand{\End}{\mathrm{End}}

\newcommand{\fVec}{\mathsf{fVect}}

\newcommand{\cVec}{\mathsf{cVect}}

\nc{\CV}{\mathbf{C}}

\begin{document}

\title[Post-groups, (Lie-)Butcher groups and the Yang-Baxter equation]{Post-groups, (Lie-)Butcher groups and the Yang-Baxter equation}

\author{Chengming Bai}
\address{Chern Institute of Mathematics and LPMC, Nankai University,
Tianjin 300071, China}
\email{baicm@nankai.edu.cn}

\author{Li Guo}
\address{Department of Mathematics and Computer Science,
         Rutgers University,
         Newark, NJ 07102}
\email{liguo@rutgers.edu}

\author{Yunhe Sheng}
\address{Department of Mathematics, Jilin University, Changchun 130012, Jilin, China}
\email{shengyh@jlu.edu.cn}

\author{Rong Tang}
\address{Department of Mathematics, Jilin University, Changchun 130012, Jilin, China}
\email{tangrong@jlu.edu.cn}


\begin{abstract}
The notions of a post-group and a pre-group are
introduced as a unification and enrichment of several group
structures appearing in diverse areas from numerical integration
to the Yang-Baxter equation.
First the Butcher group from numerical integration on Euclidean spaces and the $\huaP$-group of an operad $\huaP$ naturally admit a pre-group structure. Next a relative Rota-Baxter operator on a group naturally splits the group structure to a post-group structure. Conversely, a post-group gives rise to a relative Rota-Baxter operator on the sub-adjacent group. Further a post-group gives a braided group and a solution of the Yang-Baxter equation. Indeed the category of post-groups is isomorphic to the category of braided groups and the category of skew-left braces. Moreover a post-Lie group differentiates to a post-Lie algebra structure on the vector space of left invariant vector fields, showing that post-Lie groups are the integral objects of post-Lie algebras. Finally, post-Hopf algebras and post-Lie Magnus expansions are utilized to study the formal integration of post-Lie algebras. As a byproduct, a post-group structure is explicitly determined on the Lie-Butcher group from numerical integration on manifolds.

\end{abstract}

\subjclass[2010]{
22E60, 
16T25, 
17B38, 
65L99, 
16T05, 
18M60
}

\keywords{post-group, pre-group, (Lie-)Butcher group,  operad, relative Rota-Baxter operator, Yang-Baxter equation, skew-left brace, post-Lie algebra, formal integration}

\maketitle

\vspace{-0.6cm}

\tableofcontents

\allowdisplaybreaks

\section{Introduction}\mlabel{sec:intr}

Several group structures have arisen from diverse areas, from
(Lie-)Butcher groups in numerical integration and $\calp$-groups
of operads, to braided groups and braces for the Yang-Baxter
equation. We show that these groups have a common enrichment to what we call a post-group or
a pre-group. Differentiation of post-Lie groups gives post-Lie algebras, providing a linear structure for these group structures.
\vspace{-.2cm}

\subsection{(Lie-)Butcher groups from numerical integration}
The purpose of  numerical integration algorithms on the Euclidean space $\mathbb R^n$ is to find approximate solutions of the systems of differential equations
\vspace{-.3cm}
\begin{eqnarray} \notag
\frac{\dM y_i(x)}{\dM x}=f_i(y_1(x),\ldots,y_n(x)),\,\,\,\,y_i(x_0)=y_i{_0},\,\,\,\,i=1,2,\ldots,n.
\end{eqnarray}
In his algebraic study of the Runge-Kutta method in numerical
integration algorithms, Butcher \mcite{Bu} invented  an infinite-dimensional real Lie group, which is
called the Butcher group \mcite{HW} and plays an important role in
the  study of algebraic order conditions \mcite{CHV} of the
Runge-Kutta method. As a striking coincidence, the Butcher group
is also the character group of the Connes-Kreimer Hopf algebra
\mcite{CK} of the rooted trees in their study of renormalization
of quantum field theory. The explicit relationship between the
Butcher group and the renormalization in quantum field theory was
interpreted in \mcite{Br}. A similar group structure is the
$\calp$-group of an operad $\calp$~\mcite{CL2}.

To study the approximate solutions of differential equations on a homogeneous space,
Munthe-Kaas introduced the Lie-Butcher series and the Lie-Butcher group in \mcite{Mu,MF}.
Moreover, the Lie-Butcher group is the character group of the Hopf algebra \mcite{MW} of the planar rooted trees.
\vspace{-.2cm}

\subsection{The Yang-Baxter equation, skew-left braces and relative Rota-Baxter operators}
The Yang-Baxter equation grew out of Yang's study of exact
solutions of many-body problems \mcite{Ya} and Baxter's work of
the eight-vertex model \mcite{Bax}.  The equation is closely related to
many fields in mathematics and physics, in particular,
quasi-triangular Hopf algebras and quantum groups~\mcite{CPr}.
Due to the complexity in solving the Yang-Baxter equation, Drinfeld suggested to study  set-theoretical solutions of the Yang-Baxter equation \mcite{Dr}. Then Etingof-Schedler-Soloviev \mcite{ESS} and  Lu-Yan-Zhu \mcite{LYZ} studied the structure groups of such solutions.
Recently, the finite quotients of these structure groups were studied in \cite{LV}. See the comprehensive review \mcite{Ta} on the classification of set-theoretical solutions of the Yang-Baxter equation.

To construct solutions of the Yang-Baxter equation, braces were introduced by Rump~\mcite{Ru} as a generalization of radical rings. Further studies were carried out in~\cite{CJO2,CJO,G,Sm18}.
Recently, braces were generalized to skew-left braces by Guarnieri-Vendramin \mcite{GV} to construct
non-degenerate and not necessarily involutive  solutions of the Yang-Baxter equation. The nilpotency and factorization of skew-left braces were investigated in \mcite{CSV,JKVV-1}. Moreover, the radical and weight of skew-left braces were applied to study the structure groups of solutions of the Yang-Baxter equation \mcite{JKVV}.

The concepts of (relative) Rota-Baxter operators on groups were introduced in \mcite{GLS} in their studies of integration of (relative) Rota-Baxter operators on Lie algebras, which are the operator form of the classical Yang-Baxter equation~\mcite{STS}.
More recently, Bardakov-Gubarev studied the relationship between Rota-Baxter operators on groups and  skew-left braces, and constructed  solutions of the Yang-Baxter equation by means of Rota-Baxter operators on groups \mcite{BG}.
\vspace{-.2cm}

\subsection{Post-Lie algebras}
Post-Lie algebras  were introduced by Vallette \mcite{Val} from his study of Koszul duality of operads.
Munthe-Kaas and his coauthors found that post-Lie algebras also naturally appear in differential geometry and numerical integration on  manifolds \mcite{ML}.
Recently, post-Lie algebras are studied from different aspects including constructions of nonabelian generalized Lax pairs,  classifications of post-Lie algebras on certain Lie algebras, Poincar\'e-Birkhoff-Witt theorems, factorization theorems and regularity structures in singular stochastic partial differential equations \mcite{BGN,Bur,BK,Dotsenko,EMM}.

A pre-Lie algebra is a post-Lie algebra in which the Lie algebra is abelian and is important on its own right~\mcite{Ba,Bur-1,Ma}. The relation between pre-Lie algebras (pre-Lie rings) and braces is a fundamental problem and attracts sustained interest.  Rump first gave some connections between $\mathbb R $-braces and pre-Lie algebras in \cite{Ru14}. A simple algebraic formula  for the correspondence between finite right nilpotent $\mathbb F_p$-braces and finite nilpotent pre-Lie algebras  over the field $\mathbb F_p$ is obtained by Smoktunowicz \mcite{Sm22a}, which  agrees with the correspondence using Lazard's correspondence between finite  $\mathbb F_p$-braces and pre-Lie algebras proposed by Rump in \cite{Ru14}. Iyudu obtained a graded structure from a filtered brace in \mcite{Iy}, which turns out to be a pre-Lie algebra.
Integration of pre-Lie algebras was studied by
Dotsenko-Shadrin-Vallette \mcite{DSV} in the deformation theory and twisting procedure of algebraic structures.
Moreover, Mencattini-Quesney-Silva \mcite{MQS} used the framework
of $D$-algebras to study the integration problems of post-Lie
algebras. However, no group structures were given for pre-Lie
algebras or post-Lie algebras directly.
\vspace{-.2cm}

\subsection{Tying together by post-(Lie) groups}
In this paper, we show that all the aforementioned important group structures have a common enrichment to the notions of a post-group and, as a special case, a pre-group. The notions have a second multiplication in addition to the default group operation, thus promising finer understanding of the structures. As the name suggests, post-Lie algebras are the linear structures via differentiation for post-Lie groups, thereby providing the other groups with a linear structure which is otherwise unavailable.

First of all, the  Butcher group and the $\huaP$-group of an operad $\huaP$ have  natural pre-group structures.
Further a relative Rota-Baxter operator  on  groups gives rise to a matched pairs of groups as well as a post-group. Conversely, a post-group also gives rise to a relative Rota-Baxter operator on the sub-adjacent group. Consequently, a post-group as well as a relative Rota-Baxter operator, gives a solution of the Yang-Baxter equation. To be more specific, we obtain an explicit solution of the  Yang-Baxter equation on the Butcher group.
Post-groups also build a bridge between  Rota-Baxter operators and  skew-left braces. Moreover, we prove that the category of post-groups is isomorphic to the category of skew-left braces.

Post-Lie groups are the global objects corresponding to post-Lie algebras, in the sense that one obtains a post-Lie algebra from a post-Lie group via differentiation. To further justify the terminology of post-groups, we introduce the notion of a connected complete post-Lie algebra, and show that any connected complete post-Lie algebra can be formally integrated to a post-group by using connected complete post-Hopf algebras and post-Lie Magnus expansions.  Whether a post-Lie algebra can always be integrated to a post-Lie group, i.e. whether Lie's Third Theorem holds for post-Lie algebras, remains an open problem that is worth studying in the future.

In short, post- and pre-groups provide a richer structure for the (Lie-)Butcher groups and $\calp$-groups, and give a purely algebraic structure for relative Rota-Baxter operators, braided groups and skew-left braces.
These connections can be illustrated in the following diagram.
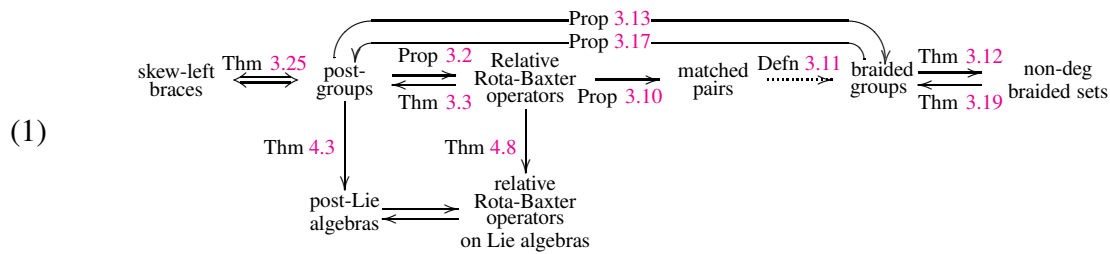
\vspace{-.2cm}

\begin{equation}
    \label{eq:bigdiag}
    \begin{split}
        \xymatrix@1{
            {\begin{subarray}{c} \text{skew-left}\\ \text{braces}
            \end{subarray}} \ \  \ar@2{<->}[r]^{\rm Thm~\mref{functor}}& \ \ {\begin{subarray}{c} \text{\ \ } \\  \text{post-} \\ \text{groups} \end{subarray}}\ \   \ar@<.3ex>[r]^{\rm Prop~ \mref{post-group-to-RB}}
\ar@<4pt> `u[r] `[rrr]|{\text{Prop~\mref{pgybe}}} [rrr]
            \ar_{\text{Thm~\mref{thm:diffpL}}}[d]&
            \ \ {\begin{subarray}{c} \text{Relative} \\\text{Rota-Baxter}\\ \text{operators}
            \end{subarray}}\ \  \ar@<.4ex>[l]^{\rm {Thm}~\mref{thm:RBPost}}
            \ar_{\rm {Prop}~\mref{rrbmp}}[r]
            \ar_{\text{Thm~\mref{thm:diffRB}}}[d]& \ \ {\begin{subarray}{c} \text{\ \ } \\ \text{matched} \\ \text{pairs}\end{subarray}} \ \ \ar@{.>}^{\text{ \  Defn~\mref{de:bg}}}[r] &\ \  {\begin{subarray}{c} \text{braided} \\ \text{groups}\  \end{subarray}}
            \ar@<.5ex>[r]^{\rm Thm~\mref{bgybe}}
\ar@<4pt> `u[l] `[lll]|{\text{Prop~\mref{ybepg}}} [lll]
& \ \ \  {\begin{subarray}{c} \text{non-deg} \\ \text{braided sets} \end{subarray}} \ar@<.4ex>[l]^{\rm {Thm}~\mref{th:bg}}
            \\
            & {\begin{subarray}{c} \text{post-Lie} \\ \text{algebras} \end{subarray}}\ar@<.3ex>[r]^{\rm } & {\begin{subarray}{c} \text{relative} \\ \text{Rota-Baxter} \\ \text{operators} \\
                    \text{on Lie algebras} \end{subarray}} \ar@<.4ex>[l]^{}&
                & &
        }
    \end{split}
\end{equation}
Here the dotted arrow means extra conditions are needed for the implication.
There is not yet a suitable direct linear structure for any of the skew-left braces, braided groups or non-degenerated braided sets in the first row of the diagram. Thus another advantage of the post-group is that through post-Lie algebra, it provides the group objects in the first row with a linear structure.

\subsection{Outline of the paper}
As a brief outline, in Section \mref{post-group},  we introduce the notion of  post-groups, which contain pre-groups as special cases. The Butcher group and the $\huaP$-group of an operad $\huaP$ naturally enjoy pre-group structures. In Section~\mref{sec:related}, we establish the connections of post-groups with relative Rota-Baxter operators on groups, set-theoretical solutions of the Yang-Baxter equation and skew-left braces.
In Section \mref{post-group-diff}, the notion of post-Lie groups is obtained by adding smooth structures on post-groups, and is shown to give a post-Lie algebra via differentiation. In Section \mref{post-group-inte}, we give the notion of a complete post-Lie algebra, which is a generalization of a nilpotent post-Lie algebra. Moreover, we use complete post-Hopf algebras and the post-Lie Magnus expansion to establish the formal integration theory of connected complete post-Lie algebras. As an application, we give explicit descriptions of the Lie-Butcher groups.

\smallskip
\noindent
{\bf Notations. } In this paper we fix a field $\bk$ of characteristic zero as the base field over which to take vector spaces, linear maps and tensor products, unless otherwise specified.

\vspace{-.1cm}

\section{Post-groups, pre-groups, Butcher groups and operads}\mlabel{post-group}

In this section, we first introduce the notion of a post-group, which gives rise to a new group, called the sub-adjacent group. A pre-group is defined as an abelian post-group. Then we show that there are underlying pre-groups of Butcher groups. Finally we show that the $\huaP$-group associated to an operad $\huaP$ also admits a pre-group structure.

\vspace{-.1cm}

\subsection{Post-groups and pre-groups}
We begin with our basic notion.
\begin{defi}
A {\bf post-group} is a group $(G,\cdot)$ equipped with a multiplication $\rhd$ on $G$ such that
\begin{enumerate}
\item for each $a\in G$, the left multiplication $$L^\rhd_a:G\to G, \quad L^\rhd_a b= a\rhd b \tforall b\in G,$$
is an automorphism of the group $(G,\cdot)$, that is,
\begin{equation}
    \mlabel{Post-2}a\rhd (b\cdot c)=(a\rhd b)\cdot(a\rhd c)\tforall a, b, c\in G;
\end{equation}
\item the following ``weighted" associativity holds,
    \begin{equation}
        \mlabel{Post-4}\big(a\cdot(a\rhd b)\big)\rhd c=a\rhd (b\rhd c)\tforall a, b, c\in G.
    \end{equation}
\end{enumerate}
 \end{defi}

Then we first have the following basic properties.
\begin{lem} Let $e$ be the unit for $\cdot$ in a post-group $(G,\cdot,\rhd)$. Then for all $a\in G$,  we have
  \begin{eqnarray}\mlabel{Post-1}a\rhd e&=&e,\\
 \mlabel{Post-3}e\rhd a&=&a.
\end{eqnarray}
\end{lem}
\begin{proof}
\meqref{Post-1} follows since, for each $a\in G$, $L^\rhd_a$ is an automorphism of the group $(G,\cdot)$.

\meqref{Post-4} gives
$$
e\rhd (e\rhd a)=\big(e\cdot(e\rhd e)\big)\rhd a=e\rhd a\tforall a\in G.
$$
Since $L^\rhd_e$ is also bijective, it gives $e\rhd a=a.$
\end{proof}

Examples of post-groups will be provided later, by (Lie-)Butcher groups, operads, relative Rota-Baxter operators as well as skew-left braces.

\begin{defi}
A {\bf homomorphism} of post-groups  from $(G,\cdot_G,\rhd_G)$ to $(H,\cdot_H,\rhd_H)$ is a map $\Psi:G\to H$ that preserves the operations $\cdot$ and $\rhd$:
\begin{equation}
\mlabel{Post-homo-1}\Psi(a\cdot_G b)=\Psi(a)\cdot_H \Psi(b), \quad
\Psi(a\rhd_G b)=\Psi(a)\rhd_H \Psi(b)\tforall a, b\in G.
\end{equation}
\end{defi}

Post-groups and their homomorphisms form a category, denoted by $\PG$. Other than the group $(G,\cdot)$, a
post-group naturally gives rise to another group on the underlying set.

\begin{thm}\mlabel{pro:subad}
Let $(G,\cdot,\rhd)$ be a post-group. Define $\circ:G\times G\to G$     by
\begin{eqnarray}
\mlabel{eq:subad-com}a\circ b&=&a\cdot(a\rhd b)\tforall a,b\in G.
\end{eqnarray}
\begin{enumerate}
    \item \mlabel{it:subad1}
Then $(G,\circ)$ is a group with $e$ being the unit, and the inverse map $\dagger:G\to G$ given by
\begin{eqnarray}  \notag
 \mlabel{eq:subad-inv}a^\dagger&:=&(L^\rhd_a)^{-1}(a^{-1}) \tforall a\in G.
\end{eqnarray}
The   group $G_\rhd:=(G,\circ)$ is called the {\bf sub-adjacent  group} of the post-group $(G,\cdot,\rhd)$.
\item The left multiplication $L^\rhd:G\to \Aut(G)$ is an action of the group $(G,\circ)$  on the group $(G,\cdot)$.
\mlabel{it:subad2}
\item
    Let $\Psi: (G,\cdot_G,\rhd_G) \to (H,\cdot_H,\rhd_H)$ be a  homomorphism of post-groups. Then  $\Psi$ is a homomorphism of the sub-adjacent groups from $(G,\circ_G)$ to $(H,\circ_H)$.
\mlabel{it:subad3}
\end{enumerate}
\end{thm}

\begin{proof}
\meqref{it:subad1}
For all $a,b,c\in G$, we have
\begin{eqnarray*}
(a\circ b)\circ c=(a\cdot(a\rhd b))\circ c=(a\cdot(a\rhd b))\cdot ((a\cdot(a\rhd b))\rhd c)\stackrel{\meqref{Post-4}}{=}(a\cdot(a\rhd b))\cdot (a\rhd (b\rhd c)).
\end{eqnarray*}
On the other hand, we have
\begin{eqnarray*}
a\circ (b\circ c)=a\circ (b\cdot(b\rhd c))=a\cdot (a\rhd(b\cdot(b\rhd c)))
\stackrel{\meqref{Post-2}}{=}a\cdot ((a\rhd b)\cdot (a\rhd(b\rhd c))).
\end{eqnarray*}
Thus the associativity of $\circ$ follows from the associativity of $\cdot$.

For $a\in G$, we have $ a\circ e=a\cdot (a\rhd
e)\stackrel{\meqref{Post-1}}{=}a. $ Thus, we deduce that $e$ is a
right unit of the multiplication $\circ$. Since
\begin{eqnarray}\mlabel{reversible}
a\rhd a^\dagger=L_a^\rhd (L_a^\rhd)^{-1}(a^{-1})=a^{-1}\tforall a\in G,
\end{eqnarray}
we obtain
$
a\circ a^\dagger=a\cdot (a\rhd a^\dagger)=e,
$
showing that $a^\dagger$ is a  right inverse of $a$. Thus $(G,\circ)$ is a group.

\noindent
\meqref{it:subad2}
Then \meqref{Post-4} indicates that $L^\rhd:G_\rhd\to \Aut(G)$ is a homomorphism of groups.

\noindent\meqref{it:subad3}
    The conclusion follows from
\begin{eqnarray}\label{BG-homo-1}
\hspace{1cm}        \Psi(a\circ_G b)\stackrel{\meqref{eq:subad-com}}{=}\Psi\big(a\cdot_G(a\rhd_G b)\big)\stackrel{\meqref{Post-homo-1}}{=}\Psi(a)\cdot_H\big(\Psi(a)\rhd_H \Psi(b)\big)=\Psi(a)\circ_H \Psi(b). \hspace{1.5cm} \qedhere
    \end{eqnarray}
\end{proof}

A post-Lie algebra reduces to a pre-Lie algebra if the Lie algebra is abelian. From this perspective, we introduce the notion of a pre-group as an abelian post group.

\begin{defi}
A post-group    $(G,\cdot,\rhd)$ is called a {\bf pre-group} if the group  $(G,\cdot)$ is an abelian group.
Then the group $(G,\circ)$ is again called the {\bf sub-adjacent group} of the pre-group.
\end{defi}

\subsection{Pre-groups and Butcher groups}

We now show that the Butcher groups are the sub-adjacent groups of certain pre-groups.

Let $\huaT$ be the set of isomorphism classes of rooted trees:
\[
    \huaT= \Big\{\begin{array}{c}
        \scalebox{0.6}{\ab}, \scalebox{0.6}{\aabb},
        \scalebox{0.6}{\aababb}, \scalebox{0.6}{\aaabbb},\scalebox{0.6}{\aabababb},
        \scalebox{0.6}{\aaabbabb}=
        \scalebox{0.6}{\aabaabbb}, \scalebox{0.6}{\aaababbb}, \scalebox{0.6}{\aaaabbbb},\scalebox{0.6}{\aababababb},\scalebox{0.6}{\aaabbababb}=\scalebox{0.6}{\aabaabbabb}=\scalebox{0.6}{\aababaabbb},\scalebox{0.6}{\aaababbabb}=\scalebox{0.6}{\aabaababbb},\scalebox{0.6}{\aaabbaabbb},\scalebox{0.6}{\aaabababbb},\scalebox{0.6}{\aaaabbbabb}=\scalebox{0.6}{\aabaaabbbb},\scalebox{0.6}{\aaaababbbb},\scalebox{0.6}{\aaaabbabbb}=\scalebox{0.6}{\aaabaabbbb},\scalebox{0.6}{\aaaaabbbbb},\ldots
            \end{array}
            \Big\}.
\]
A rooted forest is a concatenation of trees. So the rooted forests together with the empty tree form the free monoid generated by $\huaT$. A {\bf cut} $c$ of a tree $\omega$ is a subset of the edges of the tree. We denote  by $\omega \setminus c$ the forest that remains when the edges of $c$ are removed from the tree $\omega$. The cut $c$ is called {\bf admissible} if any oriented
path in the tree meets at most one cut edge. For such an admissible cut $c$, the tree of $\omega \setminus c$
that contains the root of $\omega$ is denoted by $R^c(\omega)$ and the collection of the other
rooted trees  of $\omega \setminus c$ is denoted by $P^c(\omega)$. Moreover, we denote by $C(\omega)$ the set of all cuts of $\omega$ and $AC(\omega)$  the set of all admissible cuts of $\omega$ respectively.

Let $\huaH_{\BCK}$ be the free unitary $\bk$-commutative algebra generated by $\huaT$. The well-known {\bf Butcher-Connes-Kreimer Hopf algebra} \mcite{Bu,CK} is the algebra $\huaH_{\BCK}$  equipped with

\begin{itemize}
  \item the coproduct $\Delta_{\BCK}:\huaH_{\BCK}\lon \huaH_{\BCK}\otimes \huaH_{\BCK}$ defined by
\vspace{-.1cm}
\begin{eqnarray}\mlabel{co-product}
\Delta_{\BCK}(\omega)=\omega\otimes 1+\sum_{c\in AC(\omega)}P^c(\omega)\otimes R^c(\omega) \tforall \omega\in \huaT,
\vspace{-.1cm}
\end{eqnarray}
and extended to a unital algebra homomorphism,
\item the counit  $\varepsilon:\huaH_{\BCK}\lon \bk$ defined by
$
\varepsilon(\omega)=0, $ for all $\omega\in \huaT,
$
and extended to a unital algebra homomorphism.
\item  the antipode $S:\huaH_{\BCK}\lon \huaH_{\BCK}$, which  in this case is an algebra homomorphism,  defined by
\vspace{-.2cm}
\begin{eqnarray}\mlabel{antipode}
S(\omega)=\sum_{c\in C(\omega)}(-1)^{|c|+1}\omega \setminus c \tforall \omega\in \huaT,
\vspace{-.1cm}
\end{eqnarray}
where $|c|$ is the cardinality of the cut $c$.
\end{itemize}

Let $A$  be a unitary $\bk$-commutative algebra. We set $\huaT^+=\huaT\cup \{\emptyset\}$ and denote by
\vspace{-.1cm}
$$\huaB_{A}\coloneqq \{a:\huaT^+\lon A\,|\,a(\emptyset)=1\}.$$
We define an abelian group structure on $\huaB_{A}$ as follows:
\begin{enumerate}\item
({\em group multiplication}) for $a,b\in \huaB_{A}$, \begin{eqnarray*}
(a\cdot b)(\emptyset)=1,~(a\cdot b)(\omega)=a(\omega)+b(\omega) \tforall \omega\in \huaT,
\end{eqnarray*}
\item ({\em identity element}) $ e(\emptyset)=1,~e(\omega)=0,$ for
all $\omega\in \huaT, $
\item ({\em inverse element}) for $a\in \huaB_{A}$, $
a^{-1}(\emptyset)=1,~a^{-1}(\omega)=-a(\omega),$ for all
$\omega\in \huaT. $
\end{enumerate}
Moreover, for all $a,b\in \huaB_{A}$,  define a multiplication $\rhd:\huaB_{A}\times \huaB_{A}\lon \huaB_{A}$ by
\begin{eqnarray}
\mlabel{butcher-pre-1}(a\rhd b)(\emptyset)&=&1,\\
\mlabel{butcher-pre-2}(a\rhd b)(\omega)&=&\sum_{c\in AC(\omega)}a(P^c(\omega))b(R^c(\omega)), \text{with}\,\, a(P^c(\omega))=\prod_{\tau_i\in P^c(\omega)}a(\tau_i)
 \tforall \omega\in \huaT.
\end{eqnarray}
\vspace{-.5cm}
\begin{thm}\mlabel{thm:butchergp}
  With the above notations, the triple $(\huaB_{A},\cdot,\rhd)$  is a  pre-group, whose sub-adjacent group is exactly the  character group $\Hom_{\alg}(\huaH_{\BCK}, A)$.
\end{thm}
\begin{proof}
For all $a,b,c\in \huaB_{A}$, we have
\begin{eqnarray*}
(a\rhd(b\cdot c))(\emptyset)&=&1=((a\rhd b)\cdot (a\rhd c))(\emptyset),\\
(a\rhd(b\cdot c))(\omega)&=&\sum_{c\in AC(\omega)}a(P^c(\omega))(b\cdot c)(R^c(\omega))\\
                         &=&\sum_{c\in AC(\omega)}a(P^c(\omega))\big(b(R^c(\omega))+c(R^c(\omega))\big)\\
                         &=&\sum_{c\in AC(\omega)}\big(a(P^c(\omega))b(R^c(\omega))+a(P^c(\omega))c(R^c(\omega))\big)\\
                         &=&((a\rhd b)\cdot (a\rhd c))(\omega).
\end{eqnarray*}
Thus, we obtain $a\rhd(b\cdot c)=(a\rhd b)\cdot (a\rhd c)$.

For all $a,b,c\in \huaB_{A}$, we have
\vspace{-.1cm}
\begin{eqnarray*}
\big(a\rhd(b\rhd c)\big)(\emptyset)=1=\big((a\cdot (a\rhd b))\rhd c\big)(\emptyset).
\vspace{-.1cm}
\end{eqnarray*}
For $X\in\huaH_{\BCK}$, define
$\delta(X)=\Delta_{\BCK}(X)-X\otimes 1$. Since $\Delta_{\BCK}$ is
coassociative, we deduce
\begin{eqnarray}
(\Id\otimes \delta)\circ \delta=(\Delta_{\BCK}\otimes \Id)\circ \delta.
\end{eqnarray}
We also have
\vspace{-.1cm}
\begin{eqnarray*}
((\Id\otimes \delta)\circ \delta)(\omega)&=&\sum_{c\in AC(\omega)}\sum_{\tau\in AC(R^c(\omega))}P^c(\omega)\otimes P^{\tau}(R^c(\omega))\otimes R^{\tau}(R^c(\omega)),\\
((\Delta_{\BCK}\otimes \Id)\circ \delta)(\omega)&=&\sum_{c\in AC(\omega)}\Delta_{\BCK}(P^c(\omega))\otimes R^c(\omega)=\sum_{\omega}\omega'\otimes \omega''\otimes \omega'''.
\vspace{-.1cm}
\end{eqnarray*}
Moreover, for $\omega\in \huaT$, we have
\vspace{-.1cm}
\begin{eqnarray*}
\big(a\rhd(b\rhd c)\big)(\omega)&=&\sum_{c\in AC(\omega)}\sum_{\tau\in AC(R^c(\omega))}a(P^c(\omega))b(P^{\tau}(R^c(\omega))) c(R^{\tau}(R^c(\omega))),\\
\big((a\cdot (a\rhd b))\rhd c\big)(\omega)&=&\sum_{\omega}a(\omega') b(\omega'') c(\omega''').
\vspace{-.1cm}
\end{eqnarray*}
Thus, we obtain $a\rhd(b\rhd c)=(a\cdot (a\rhd b))\rhd c$. Therefore, $(\huaB_{A},\cdot,\rhd)$  is a  pre-group.

The sub-adjacent group is the character group on   $\Hom_{\alg}(\huaH_{\BCK},A)$ since
\vspace{-.1cm}
\begin{eqnarray*}
(a\circ b) (1) &\coloneqq& (a\cdot (a\rhd b))(\emptyset)=1,\\
(a\circ b)(\omega)&\coloneqq& (a\cdot (a\rhd b))(\omega)=a(\omega)+\sum_{c\in AC(\omega)}a(P^c(\omega))b(R^c(\omega)),
 \tforall \omega\in \huaT. \hspace{1.5cm} \qedhere
\end{eqnarray*}
\end{proof}
\vspace{-.1cm}

Define the $A$-{\bf Butcher group} (see \mcite{HW} for the cases of $A=\mathbb{R}$ and $\mathbb{C}$) to be the set $$\huaB_{A}\coloneqq \{a:\huaT^+\lon A\,|\,a(\emptyset)=1\}$$
together with the group multiplication $\circ:\huaB_{A}\times \huaB_{A}\lon \huaB_{A}$ given by
\vspace{-.1cm}
\begin{eqnarray*}
\mlabel{butcher-1}(a\circ b)(\emptyset)&=&1,\\
\mlabel{butcher-2}(a\circ b)(\omega)&=&a(\omega)+\sum_{c\in AC(\omega)}a(P^c(\omega))b(R^c(\omega)), \text{with}\,\, a(P^c(\omega))=\prod_{\tau_i\in P^c(\omega)}a(\tau_i)
 \tforall \omega\in \huaT.
\end{eqnarray*}
Thus we immediately have the following results.
\vspace{-.1cm}
\begin{cor}\mlabel{Butcher-G}
\begin{enumerate}
    \item
The {\bf complex Butcher group} $\huaB_{\mathbb C}$ $($resp. {\bf real Butcher group} $\huaB_{\mathbb R}$$)$ is the sub-adjacent group of the pre-group $(\huaB_{\mathbb C},\cdot,\rhd)$ $($resp. $(\huaB_{\mathbb R},\cdot,\rhd)$$)$.

\item
Taking $A$ to be the Laurent series $\mathbb C[[Z,Z^{-1}]$, the character group $\Hom_{\alg}(\huaH_{\BCK},\mathbb C[[Z,Z^{-1}])$ is the sub-adjacent group of the pre-group $(\huaB_{\mathbb C[[Z,Z^{-1}]},\cdot,\rhd)$.
\item
Taking $A$ to be the finite field $\mathbb F_p$, here $p$ is an odd prime number, the character group $\Hom_{\alg}(\huaH_{\BCK},\mathbb F_p)$ is the sub-adjacent group of the pre-group $(\huaB_{\mathbb F_p},\cdot,\rhd)$.
\end{enumerate}
\end{cor}

\begin{ex}\label{truncation-finite-pre-groups}
Let $\huaT^n$ be the set of isomorphism classes of rooted trees with at most $n$ vertices. We denote  the set of all
maps $\{a:\huaT^n\cup \{\emptyset\}\lon \mathbb F_p\,|\,a(\emptyset)=1\}$ by $\huaB_{\mathbb F_p}^n$, which is closed under  the pre-group structure on $\huaB_{\mathbb F_p}$. Thus, $(\huaB_{\mathbb F_p}^n,\cdot,\rhd)$ is a sub pre-group of $(\huaB_{\mathbb F_p},\cdot,\rhd)$. So for each $n$, $(\huaB_{\mathbb F_p}^n,\cdot,\rhd)$ is a finite pre-group.
\end{ex}

\subsection{Pre-groups and  operads}
We now show that the $\huaP$-group associated to an operad $\huaP$  admits a pre-group. We recall the notions and an example for later use.

\begin{defi}\mlabel{operad}
An operad is a collection $\huaP=\{\huaP(n)\}_{n=1}^{+\infty}$ of right $\bk[\mathbb S_{n}]$-modules together with a family of linear  morphisms
\begin{eqnarray}
\gamma:\huaP(n)\otimes\huaP(m_1)\otimes \cdots\otimes \huaP(m_n)\lon \huaP(m_1+\cdots+m_n) \tforall n,m_1,\ldots,m_n\ge 1,
\end{eqnarray}
satisfying the following three axioms.
\begin{enumerate}\item
({\em Associativity})
For all $f\in\huaP(n),g_1\in\huaP(m_1),\ldots,g_n\in\huaP(m_n),h_{11}\in\huaP(k_1),\ldots,h_{nm_n}\in\huaP(k_{m_1+\cdots+m_n})$,
\begin{equation}\mlabel{operad-1}
\begin{split}
&\gamma(\gamma(f;g_1,\ldots,g_n);h_{11},\ldots,h_{1m_1},h_{21},\ldots,h_{2m_2},\ldots,h_{n1},\ldots,h_{nm_n})\\
=&\gamma(f;\gamma(g_1;h_{11},\ldots,h_{1m_1}),\gamma(g_2;h_{21},\ldots,h_{2m_2}),\ldots,\gamma(g_n;h_{n1},\ldots,h_{nm_n})).
\end{split}
\end{equation}
\item ({\em Unitality})
There is a unity element $\Id\in\huaP(1)$ such that for all $f\in\huaP(n),~n\ge 1,$
\begin{eqnarray}
\mlabel{operad-2}\gamma(\Id;f)=\gamma(f;\underbrace{\Id,\ldots,\Id}_n)=f.
\end{eqnarray}

\item ({\em Equivariance})
For all $f\in\huaP(n),g_1\in\huaP(m_1),\ldots,g_n\in\huaP(m_n)$ and $\sigma\in \mathbb S_{n},~\sigma_1\in \mathbb S_{m_1},\ldots,\sigma_n\in \mathbb S_{m_n}$,
\begin{eqnarray}
\mlabel{operad-3}\gamma(f\sigma;g_1,\ldots,g_n)&=&\gamma(f;g_{\sigma^{-1}(1)},\ldots,g_{\sigma^{-1}(n)})\sigma_{m_{1},\ldots,m_{n}},\\
\mlabel{operad-4}\gamma(f;g_1\sigma_1,\ldots,g_n\sigma_n)&=&\gamma(f;g_{1},\ldots,g_{n})(\sigma_1\times\ldots\times \sigma_n),
\end{eqnarray}
here $\sigma_{m_{1},\ldots,m_{n}}\in\mathbb S_{m_1+\cdots+m_n}$ is the permutation which permutes the  blocks $\pi_i=\{k_1+\cdots+k_{i-1}+1,\ldots,k_1+\cdots+k_{i-1}+k_i\},~i=1,\ldots,n$ according to $\sigma$ and $\sigma_1\times\cdots\times \sigma_n\in \mathbb S_{m_1+\cdots+m_n}$ is  the permutation which acts like $\sigma_i$ on the block $\pi_i,~i=1,\ldots,n$.
\end{enumerate}
\end{defi}

\begin{ex}
The operad of commutative algebras is given by $Com=\{Com(n)\}_{n=1}^{+\infty}$ with
$
Com(n)=\bk,
$
which is equipped with the trivial $\bk[\mathbb S_{n}]$-module structure. The linear map
$$\gamma:Com(n)\otimes Com(m_1)\otimes \cdots\otimes Com(m_n)\lon Com(m_1+\cdots+m_n) \tforall n,m_1,\ldots,m_n\ge 1,$$
is defined by
$\gamma(k;k_1,\ldots,k_n)=kk_1\cdots k_n.$
Moreover, $1_\bk\in \bk$ is the unity element.
\end{ex}

\begin{defi}
 Let  $\huaP$ and $\huaQ$ be operads. A homomorphism from $\huaP$ to $\huaQ$ is a sequence $\{\phi_n:\huaP(n)\lon\huaQ(n)\}_{n=1}^{+\infty}$ of $\bk[\mathbb S_{n}]$-module homomorphisms satisfying the following relations:
 \begin{eqnarray}
\mlabel{homo-1}\phi_1(\Id_\huaP)&=&\Id_\huaQ,\\
\mlabel{homo-2}\phi_{m_1+\cdots+m_n}\big(\gamma(f;g_1,\ldots,g_n)\big)&=&\gamma\big(\phi_n(f);\phi_{m_1}(g_1),\ldots,\phi_{m_n}(g_n)\big),
\end{eqnarray}
for all
$f\in\huaP(n),g_1\in\huaP(m_1),\ldots,g_n\in\huaP(m_n)$.
\end{defi}

Let $\huaP$ be an operad. We denote by
\vspace{-.2cm}
\begin{eqnarray}
G(\huaP)\coloneqq \{\Id_\huaP\}\times\prod_{n=2}^{+\infty}\huaP(n)_{\mathbb S_{n}},
\vspace{-.1cm}
\end{eqnarray}
here $\huaP(n)_{\mathbb S_{n}}$ is the space of coinvariants which is given by $\huaP(n)_{\mathbb S_{n}}\coloneqq \huaP(n)/\{v\sigma-v| \sigma\in \mathbb S_{n},v\in\huaP(n)\}$.
Denote an element of $G(\huaP)$ by $\bar{a}=(\Id_\huaP,\overline{a_2},\ldots,\overline{a_n},\ldots)$.
For all $\bar{a},\bar{b}\in G(\huaP)$,  define a multiplication $\circ :G(\huaP)\times G(\huaP)\lon G(\huaP)$ by
\vspace{-.2cm}
\begin{eqnarray}
(\bar{a}\circ \bar{b})_n=\sum_{k=1}^{n}\sum_{t_1+\cdots+t_k=n}\overline{\gamma(b_k;a_{t_1},\ldots,a_{t_k})}.
\vspace{-.1cm}
\end{eqnarray}
\vspace{-.3cm}

\begin{thm} \mcite{CL2} \mlabel{symmetric-operad-group}
Let $\huaP$ be an operad. Then $(G(\huaP),\circ)$ is a group, called the {\bf $\huaP$-group} of $\huaP$.
\end{thm}

We define an abelian group structure on $G(\huaP)$ as follows:
\begin{itemize}\item[\rm(i)]
{\em (group multiplication)} for all $\bar{a},\bar{b}\in G(\huaP)$,
\begin{eqnarray*}
(\bar{a}\cdot \bar{b})_1=\Id_\huaP,~(\bar{a}\cdot \bar{b})_n=\overline{a_n+b_n}\tforall n=2,3,\ldots,
\end{eqnarray*}
\item[\rm(ii)] {\em (identity element)}
$
e=(\Id_\huaP,\bar{0},\ldots,\bar{0},\ldots),
$
\item[\rm(iii)] {\em (inverse element)} for all  $a\in G(\huaP)$,
$
(\bar{a}^{-1})_1=\Id_\huaP,~(\bar{a}^{-1})_n=-\overline{a_n}\tforall n=2,3,\ldots.
$
\end{itemize}
Moreover, for all $\bar{a},\bar{b}\in G(\huaP)$, we define the binary product $\rhd:G(\huaP)\times G(\huaP)\lon G(\huaP)$ by
\vspace{-.1cm}
$$(\bar{a}\rhd \bar{b})_1=\Id_\huaP, \ \ \ (\bar{a}\rhd \bar{b})_n=\sum_{k=2}^{n}\sum_{t_1+\cdots+t_k=n}\overline{\gamma(b_k;a_{t_1},\ldots,a_{t_k})}.
\vspace{-.2cm}
$$
For all $b_k\in \huaP(k), \sigma\in \mathbb S_{k},~k\ge 2$, we have
\begin{eqnarray*}
&&\sum_{t_1+\cdots+t_k=n}\gamma(b_k-b_k\sigma;a_{t_1},\ldots,a_{t_k})\\
&=&\sum_{t_1+\cdots+t_k=n}\gamma(b_k;a_{t_1},\ldots,a_{t_k})-\sum_{t_1+\cdots+t_k=n}\gamma(b_k\sigma;a_{t_1},\ldots,a_{t_k})\\
&\stackrel{\meqref{operad-3}}{=}&\sum_{t_1+\cdots+t_k=n}\gamma(b_k;a_{t_{\sigma^{-1}(1)}},\ldots,a_{t_{\sigma^{-1}(k)}})-\sum_{t_1+\cdots+t_k=n}\gamma(b_k;a_{t_{\sigma^{-1}(1)}},\ldots,a_{t_{\sigma^{-1}(k)}})\sigma_{t_1,\ldots,t_k}\\
&=&\sum_{t_1+\cdots+t_k=n}\Big(\gamma(b_k;a_{t_{\sigma^{-1}(1)}},\ldots,a_{t_{\sigma^{-1}(k)}})-\gamma(b_k;a_{t_{\sigma^{-1}(1)}},\ldots,a_{t_{\sigma^{-1}(k)}})\sigma_{t_1,\ldots,t_k}\Big).
\end{eqnarray*}
Also for $b_k\in
\huaP(k),~a_{t_i}\in\huaP(t_i),~i=1,2,\ldots,k,~\sigma_i\in
\mathbb S_{t_i}$, we have
\begin{eqnarray*}
&&\gamma(b_k;a_{t_1},\ldots,a_{t_k})-\gamma(b_k;a_{t_1}\sigma_1,\ldots,a_{t_k}\sigma_k)\\
&\stackrel{\meqref{operad-4}}{=}&\gamma(b_k;a_{t_1},\ldots,a_{t_k})-\gamma(b_k;a_{t_1},\ldots,a_{t_k})(\sigma_1\times\cdots\times \sigma_k).
\end{eqnarray*}
Thus the multiplication $\rhd$ on $G(\huaP)$ is well-defined.

Moreover, we note that
\vspace{-.1cm}
\begin{eqnarray*}
(\bar{a}\cdot (\bar{a}\rhd \bar{b}))_1&=&\Id_\huaP=(\bar{a}\circ \bar{b})_1,\\
(\bar{a}\cdot (\bar{a}\rhd \bar{b}))_n&=&\bar{a}_n+(\bar{a}\rhd \bar{b})_n=\sum_{k=1}^{n}\sum_{t_1+\cdots+t_k=n}\overline{\gamma(b_k;a_{t_1},\ldots,a_{t_k})}=(\bar{a}\circ \bar{b})_n \tforall n\geq 2.
\vspace{-.1cm}
\end{eqnarray*}

Finally, we can realize a $\huaP$-group as the sub-adjacent group of a pre-group.
\vspace{-.1cm}

\begin{thm}\mlabel{O-to-pre-group}
  With the above notations, $(G(\huaP),\cdot,\rhd)$  is a  pre-group, whose sub-adjacent group is exactly the  $\huaP$-group $(G(\huaP),\circ)$.
\end{thm}

\begin{proof}
For all $\bar{a},\bar{b},\bar{c}\in G(\huaP)$, we have
\vspace{-.1cm}
\begin{eqnarray*}
(\bar{a}\rhd(\bar{b}\cdot \bar{c}))_1&=&\Id_\huaP=((\bar{a}\rhd \bar{b})\cdot (\bar{a}\rhd \bar{c}))_1,\\
(\bar{a}\rhd(\bar{b}\cdot \bar{c}))_n
                   &=&\sum_{k=2}^{n}\sum_{t_1+\cdots+t_k=n}\overline{\gamma\big(b_k+ c_k;a_{t_1},\ldots,a_{t_k}\big)}\\
                   &=&\sum_{k=2}^{n}\sum_{t_1+\cdots+t_k=n}\overline{\gamma\big(b_k;a_{t_1},\ldots,a_{t_k}\big)}+\sum_{k=2}^{n}\sum_{t_1+\cdots+t_k=n}\overline{\gamma\big(c_k;a_{t_1},\ldots,a_{t_k}\big)}\\
                   &=&((\bar{a}\rhd \bar{b})\cdot (\bar{a}\rhd \bar{c}))_n.
\vspace{-.1cm}
\end{eqnarray*}
Thus, we obtain $\bar{a}\rhd(\bar{b}\cdot \bar{c})=(\bar{a}\rhd \bar{b})\cdot (\bar{a}\rhd \bar{c})$. For $\bar{a},\bar{b},\bar{c}\in G(\huaP)$, we have
\vspace{-.2cm}
\begin{eqnarray*}
\big(\bar{a}\rhd(\bar{b}\rhd \bar{c})\big)_1=\Id_\huaP=\big((\bar{a}\cdot (\bar{a}\rhd \bar{b}))\rhd \bar{c}\big)_1.
\vspace{-1cm}
\end{eqnarray*}
For $n\geq 2,$ we have
\vspace{-.2cm}
{\small \begin{eqnarray*}
\big(\bar{a}\rhd(\bar{b}\rhd \bar{c})\big)_n&=&\sum_{k=2}^{n}\sum_{t_1+\cdots+t_k=n}\gamma\big((\bar{b}\rhd \bar{c})_k;\bar{a}_{t_1},\ldots,\bar{a}_{t_k}\big)\\
&=&\sum_{k=2}^{n}\sum_{t_1+\cdots+t_k=n}\sum_{s=2}^{k}\sum_{l_1+\cdots+l_s=k}\overline{\gamma\big(\gamma(c_s;b_{l_1},\ldots,b_{l_s});a_{t_1},\ldots,a_{t_k}\big)}\\
&\stackrel{\meqref{operad-1}}{=}&\sum_{k=2}^{n}\sum_{t_1+\cdots+t_k=n}\sum_{s=2}^{k}\sum_{l_1+\cdots+l_s=k}\overline{\gamma\big(c_s;\gamma(b_{l_1};a_{t_1},\ldots,a_{t_{l_1}}),\ldots,\gamma(b_{l_s};a_{t_{l_1+\cdots+l_{s-1}+1}},\ldots,a_{t_k})\big)}\\
&{=}&\sum_{k=2}^{n}\sum_{s=2}^{k}\sum_{l_1+\cdots+l_s=k}\sum_{t_1+\cdots+t_k=n}\overline{\gamma\big(c_s;\gamma(b_{l_1};a_{t_1},\ldots,a_{t_{l_1}}),\ldots,\gamma(b_{l_s};a_{t_{l_1+\cdots+l_{s-1}+1}},\ldots,a_{t_k})\big)}\\
&=&\sum_{s=2}^{n}\sum_{m_1+\cdots+m_s=n}\gamma\big(\bar{c}_s;(\bar{a}\circ \bar{b})_{m_1},\ldots,(\bar{a}\circ \bar{b})_{m_s}\big)\\
&=&\big((\bar{a}\cdot (\bar{a}\rhd \bar{b}))\rhd \bar{c}\big)_n.
\end{eqnarray*}
}
Thus, we obtain $\bar{a}\rhd(\bar{b}\rhd \bar{c})=(\bar{a}\cdot (\bar{a}\rhd \bar{b}))\rhd \bar{c}$. By Theorem \mref{symmetric-operad-group}, we deduce that the left multiplication $L^\rhd_{\bar{a}}$ is a bijection for all $\bar{a}\in G(\huaP)$. Thus,   $(G(\huaP),\cdot,\rhd)$  is a  pre-group.
\end{proof}

Let $\{\phi_n:\huaP(n)\lon\huaQ(n)\}_{n=1}^{+\infty}$ be a  homomorphism from the operad $\huaP$ to $\huaQ$. Define $\Phi:G(\huaP)\lon G(\huaQ)$ by
\vspace{-.2cm}
$$
\Phi(\Id_\huaP,\bar{a}_2,\ldots,\bar{a}_n,\ldots)=(\phi_1(\Id_\huaP),\overline{\phi_2(a_2)},\ldots,\overline{\phi_n(a_n)},\ldots) \tforall \bar{a}=(\Id_\huaP,\bar{a}_2,\ldots,\bar{a}_n,\ldots)\in G(\huaP).$$
\begin{pro}\mlabel{O-to-pre-group-homo}
 With the above notations,   $\Phi$ is a homomorphism of pre-groups from  $(G(\huaP),\cdot,\rhd)$ to $(G(\huaQ),\cdot,\rhd)$.
\end{pro}

\begin{proof}
For all $\bar{a},\bar{b}\in G(\huaP)$, we have
\begin{eqnarray*}
\phi_1\big((\bar{a}\cdot \bar{b})_1\big)&=&\phi_1\big(\Id_\huaP\big)\stackrel{\meqref{homo-1}}{=}\Id_{\huaQ}=\big(\Phi(\bar{a})\cdot\Phi(\bar{b})\big)_1,\\
\overline{\phi_n\big((\bar{a}\cdot \bar{b})_n\big)}&=&\overline{\phi_n(a_n + b_n)}=\overline{\phi_n(a_n)} + \overline{\phi_n(b_n)}=\big(\Phi(\bar{a})\cdot\Phi(\bar{b})\big)_n.
\end{eqnarray*}
Thus, we deduce $\Phi(\bar{a}\cdot \bar{b})=\Phi(\bar{a})\cdot \Phi(\bar{b})$. Moreover, we have
\begin{eqnarray*}
\phi_1\big((\bar{a}\rhd \bar{b})_1\big)&=&\phi_1\big(\Id_\huaP\big)\stackrel{\meqref{homo-1}}{=}\Id_{\huaQ}=\big(\Phi(\bar{a})\rhd\Phi(\bar{b})\big)_1,\\
\overline{\phi_n\big((\bar{a}\rhd \bar{b})_n\big)}&=&\sum_{k=2}^{n}\sum_{t_1+\cdots+t_k=n}\overline{\phi_n\big(\gamma(b_k;a_{t_1},\ldots,a_{t_k})\big)}\\
                           &\stackrel{\meqref{homo-2}}{=}&\sum_{k=2}^{n}\sum_{t_1+\cdots+t_k=n}\overline{\gamma\big(\phi_k(b_k);\phi_{t_1}(a_{t_1}),\ldots,\phi_{t_k}(a_{t_k})\big)}\\
                           &=&\big(\Phi(\bar{a})\rhd\Phi(\bar{b})\big)_n.
\end{eqnarray*}
Thus, we deduce that $\Phi(\bar{a}\rhd \bar{b})=\Phi(\bar{a})\rhd \Phi(\bar{b})$. Therefore,   $\Phi$ is a homomorphism of pre-groups from $(G(\huaP),\cdot,\rhd)$ to $(G(\huaQ),\cdot,\rhd)$.
\end{proof}

By Theorem \mref{O-to-pre-group} and Proposition \mref{O-to-pre-group-homo}, we have the following result.
\begin{thm}\mlabel{NO-to-group-cate}
The construction of pre-groups from operads is a functor from the category of
operads to the category of pre-groups.
\end{thm}
\vspace{-.2cm}

\begin{ex}
Let $\huaP$ be the operad $Com$. Then the $Com$-group $G(Com)$ is given by
\begin{eqnarray*}
(k_0=1_\bk,k_1,\ldots,k_n,\ldots)\circ (l_0=1_\bk,l_1,\ldots,l_n,\ldots)=(1_\bk,k_1+l_1,\ldots,\sum_{s=1}^{n+1}\sum_{t_1+\cdots+t_s=n+1-s\atop t_1,t_2,\cdots,t_{s}\ge 0}l_{s-1}k_{t_1}\cdots k_{t_s},\ldots,).
\end{eqnarray*}
Moreover, the  pre-group structure on the $Com$-group $G(Com)$ is given by
\vspace{-.1cm}
\begin{eqnarray*}
(1_\bk,k_1,\ldots,k_n,\ldots)\cdot (1_\bk,l_1,\ldots,l_n,\ldots)&=&(1_\bk,k_1+l_1,k_2+l_2,\ldots,k_n+l_n,\ldots),\\
(1_\bk,k_1,\ldots,k_n,\ldots)\rhd (1_\bk,l_1,\ldots,l_n,\ldots)&=&(1_\bk,l_1,l_2+2l_1k_1,\ldots,\sum_{s=2}^{n+1}\sum_{t_1+\cdots+t_s=n+1-s\atop t_1,t_2,\cdots,t_{s}\ge 0}l_{s-1}k_{t_1}\cdots k_{t_s},\ldots,).
\end{eqnarray*}
Equivalently, there is a pre-group structure on the set of $1_\bk+x\bk [[x]]$ which is given by
\begin{eqnarray*}
(1_\bk+\sum_{n=1}^{+\infty}k_nx^n)\cdot (1_\bk+\sum_{n=1}^{+\infty}l_nx^n)&=&1_\bk+\sum_{n=1}^{+\infty}(k_n+l_n)x^n,\\
(1_\bk+\sum_{n=1}^{+\infty}k_nx^n)\rhd (1_\bk+\sum_{n=1}^{+\infty}l_nx^n)&=&1_\bk+\sum_{n=1}^{+\infty}\Big(\sum_{s=2}^{n+1}\sum_{t_1+\cdots+t_s=n+1-s\atop t_1,t_2,\cdots,t_{s}\ge 0}l_{s-1}k_{t_1}\cdots k_{t_s}\Big)x^n.
\end{eqnarray*}
\end{ex}

\section{Post-groups, the Yang-Baxter equation, relative Rota-Baxter operators and skew-left braces}
\mlabel{sec:related}

In this section, we show that post-groups play a central role in several diverse notions ranging from relative Rota-Baxter operators and the Yang-Baxter equation to skew-left braces, as showing in the diagram \meqref{eq:bigdiag}.

\subsection{Post-groups and relative Rota-Baxter operators}\mlabel{post-group-rota-baxter}

We show that a relative Rota-Baxter operator naturally induces a post-group, whose sub-adjacent group is exactly the descendant group given in \mcite{GLS}. In addition, for a post-group, the identity map is naturally a relative Rota-Baxter operator on the sub-adjacent group.

In the sequel, $(G, e_G, \cdot_G)$ and $(H, e_H, \cdot_H)$ are groups, and   $\Phi:G \lon \Aut(H)$ is an action of the group $G$ on the group $H$. On $H\times G$, there is the following group multiplication
$$
(h,a)\cdot_\Phi(k,b)=(h\cdot_H \Phi(a)k,a\cdot_G b)\tforall a,b\in G,~h,k\in H.
$$
The group $(H\times G,\cdot_\Phi)$ is called the {\bf semidirect product} of $G$ and $H$, and denoted by
$H\rtimes_{\Phi}G$.
\begin{defi}\mcite{GLS}
Let $(G,\cdot_G)$ and $(H,\cdot_H)$ be groups with a group action $\Phi:G\to \Aut(H)$. A map $\huaB:H\longrightarrow G$ is called a {\bf relative Rota-Baxter operator} on $(G,\cdot_G)$ with respect to the action $\Phi$ if the following equality holds for all
$h,k\in H$,
\begin{equation}\mlabel{RRBO}
 \huaB(h)\cdot_G \huaB(k)=\huaB\big(h\cdot_H\Phi(\huaB(h))(k)\big).
\end{equation}
In particular, a  relative Rota-Baxter operator  $\huaB:G\longrightarrow G$  on a group $(G,\cdot_G)$ with respect to the adjoint action $\Ad$ is called a {\bf Rota-Baxter operator} on $G$. More precisely, $\huaB:G\longrightarrow G$ is a Rota-Baxter operator if
\vspace{-.1cm}
\begin{eqnarray}\mlabel{RBG}
\huaB(a)\huaB(b)=\huaB(a\huaB(a)b\huaB(a)^{-1}) \tforall a,b\in G.
\end{eqnarray}
\end{defi}

Let  $\huaB:H\longrightarrow G$ be a relative Rota-Baxter operator. Then it was proved in \mcite{GLS} that the multiplication $\circ_H:H\times H\lon H$ given by
\begin{eqnarray}\mlabel{eq:descen}
h\circ_H k=h\cdot_H\Phi(\huaB(h))(k) \tforall h,k\in H,
\end{eqnarray}
defines a new group structure on $H$, called the {\bf descendant group} of $\huaB$.

\begin{pro}\mlabel{post-group-to-RB}
Let $(G,\cdot,\rhd)$ be a post-group. Then   the identity map $\Id:G\lon G$ is a relative Rota-Baxter operator on the sub-adjacent group $(G,\circ)$ given in Theorem \mref{pro:subad} with respect to the action $L^\rhd$ on the group $(G,\cdot)$.
\end{pro}
\begin{proof}
By \meqref{eq:subad-com}, we get
\vspace{-.1cm}
\begin{eqnarray}
\Id(a)\circ \Id(b)=\Id(a\cdot (L^\rhd_{\Id(a)}  b)) \tforall a,b\in G,
\end{eqnarray}
which implies that $\Id:G\lon G$ is a relative Rota-Baxter operator on $(G,\circ)$ with respect to the action $L^\rhd$ on the group $(G,\cdot)$.
\end{proof}

It is well known that relative Rota-Baxter operators of nonzero weights on Lie algebras induce post-Lie algebras.
On the level of groups, we have the following result.
\begin{thm}\mlabel{thm:RBPost}
Let $\huaB:H\longrightarrow G$ be a  relative Rota-Baxter operator on a group $(G,\cdot_G)$ with respect to an action $\Phi$. Define a multiplication $\rhd:H\times H\lon H$ by
\begin{eqnarray}
h\rhd k\coloneqq \Phi(\huaB(h))(k) \tforall h,k\in H.
\end{eqnarray}
Then $(H,\cdot_H,\rhd)$ is a post-group, whose sub-adjacent group is the descendant group of $\huaB$.
\end{thm}
\begin{proof}
Let  $\huaB:H\longrightarrow G$ be a  relative Rota-Baxter
operator on a group $(G,\cdot_G)$ with respect to an action
$\Phi$. Then $(H,\circ_H)$ is a group and  $\huaB$ is a
homomorphism from the group $(H,\circ_H)$ to $(G,\cdot_G)$, where
$\circ_H$ is given by \meqref{eq:descen}. Since
$\Phi:(G,\cdot_G)\lon \Aut(H,\cdot_H)$ is an action of
$(G,\cdot_G)$ on $(H,\cdot_H)$, we deduce that
$\Phi\circ\huaB:(H,\circ_H)\to \Aut(H,\cdot_H)$ is an action of
$(H,\circ_H)$ on $(H,\cdot_H)$. Moreover, for $h\in H$, we have $L_h^{\rhd}=\Phi(\huaB(h))$. Therefore, we
deduce that $(H,\cdot_H,\rhd)$ is a post-group. The other
conclusion is obvious.
\end{proof}

\begin{rmk}
When a relative Rota-Baxter operator on a group is invertible, its inverse is a crossed homomorphism, that is, a bijective $1$-cocycle~\mcite{GLS}. So invertible
crossed homomorphisms on groups or bijective $1$-cocycles give rise to post-groups by the above theorem.
\end{rmk}

Let $G, H$ be groups and let $\huaB:H\to G$ be a map. The map
\begin{eqnarray}
\xi_\huaB:H\times G\to H\rtimes_{\Phi}G, \quad  \xi_\huaB(h,a)=(h,\huaB(h)\cdot_Ga) \tforall h\in H,~a\in G,
\end{eqnarray}
is invertible. In fact, the inverse map $\xi^{-1}:H\rtimes_{\Phi}G\to H\times G$  is given by
\begin{eqnarray}
\xi_\huaB^{-1}(h,a)=(h,\huaB(h)^{-1}\cdot_Ga) \tforall h\in H,~a\in G.
\end{eqnarray}
Transfer the group structure on $H\rtimes_{\Phi}G$ to $H\times G$, we obtain a group $(H\times G,\ast)$. So the multiplication $\ast$ is  given by
\begin{equation}
\begin{split}
&(h,a)\ast(k,b)\\
=&\xi_\huaB^{-1}(\xi_\huaB(h,a)\cdot \xi_\huaB(k,b))\\ =&\xi_\huaB^{-1}((h,\huaB(h)\cdot_Ga)\cdot (k,\huaB(k)\cdot_Gb))\\
=&\xi_\huaB^{-1}(h\cdot_H \Phi(\huaB(h)\cdot_Ga)(k),\huaB(h)\cdot_Ga\cdot_G\huaB(k)\cdot_Gb)\\
=&\Big(h\cdot_H \Phi(\huaB(h)\cdot_Ga)(k),\huaB(h\cdot_H \Phi(\huaB(h)\cdot_Ga)(k))^{-1}\cdot_G\huaB(h)\cdot_Ga\cdot_G\huaB(k)\cdot_Gb\Big).
\end{split}
\mlabel{mathched-pair}
\end{equation}
Moreover, the unit is  $e=\xi_\huaB^{-1}(e_H,e_G)=(e_H,\huaB(e_H)^{-1})$ and the inverse of $(h,a)$ is
\begin{eqnarray}\mlabel{inverse-big}
(h,a)^{\dagger}=\Big(\Phi(\huaB(h)\cdot_G a)^{-1}(h^{-1}),\huaB(\Phi(\huaB(h)\cdot_G a)^{-1}(h^{-1}))^{-1}\cdot_G(\huaB(h)\cdot_G a)^{-1}\Big).
\end{eqnarray}

\begin{pro}\mlabel{pro:factor}
With the  above notations, $(H\times G,\ast)$  has a group factorization into  subgroups $H\times\{e_G\}$ and $\{e_H\}\times G$  if and only if $\huaB:H\to G$ is a relative Rota-Baxter operator on the group $(G,\cdot_G)$ with respect to the action $\Phi$.
\end{pro}

\begin{proof}
 For all $h,k\in H$, we have
\begin{eqnarray}\mlabel{descendant-group}
(h,e_G)\ast(k,e_G)&\stackrel{\meqref{mathched-pair}}{=}&\Big(h\cdot_H \Phi(\huaB(h))(k),\huaB(h\cdot_H \Phi(\huaB(h))(k))^{-1}\cdot_G\huaB(h)\cdot_G\huaB(k)\Big).
\end{eqnarray}
Thus, if $H\times\{e_G\}$ is a subgroup of $(H\times G,\ast)$, we get
$$
\huaB(h\cdot_H \Phi(\huaB(h))(k))^{-1}\cdot_G\huaB(h)\cdot_G\huaB(k)=e_G,
$$
which implies  that $\huaB:H\to G$ is a relative Rota-Baxter operator on the group $(G,\cdot_G)$ with respect to the action $\Phi$.

Conversely, let $\huaB:H\to G$ be a relative Rota-Baxter operator.
We have $\huaB(e_H)=e_G$. By \meqref{descendant-group}, we obtain
$(h,e_G)\ast(k,e_G)\in H\times\{e_G\}$. For
$(h,e_G)\in H\times\{e_G\}$, we have
\begin{eqnarray*}
(h,e_G)^{\dagger}&\stackrel{\meqref{inverse-big}}{=}&(\Phi(\huaB(h))^{-1}(h^{-1}),\huaB(\Phi(\huaB(h))^{-1}(h^{-1}))^{-1}\cdot_G(\huaB(h))^{-1})\\
                 &=&\Big(\Phi(\huaB(h))^{-1}(h^{-1}),\Big(\huaB(h)\cdot_G\huaB\big(\Phi(\huaB(h))^{-1}(h^{-1})\big)\Big)^{-1}\Big)\\
                 &\stackrel{\meqref{RRBO}}{=}&\Big(\Phi(\huaB(h))^{-1}(h^{-1}),\Big(\huaB\big(h\cdot_H\Phi(\huaB(h))(\Phi(\huaB(h))^{-1}(h^{-1}))\big)\Big)^{-1}\Big)\\
                 &=&(\Phi(\huaB(h))^{-1}(h^{-1}),\huaB(e_H)^{-1})\\
                 &=&(\Phi(\huaB(h))^{-1}(h^{-1}),e_G).
\end{eqnarray*}
Thus, we deduce that $(h,e_G)^{\dagger}\in H\times\{e_G\}$, and $H\times\{e_G\}$ is a subgroup of $(H\times G,\ast)$.

Since $\huaB(e_H)=e_G$, for all $(e_H,a),(e_H,b)\in \{e_H\}\times G$, we have
\begin{eqnarray*}
(e_H,a)\ast(e_H,b)&\stackrel{\meqref{mathched-pair}}{=}&(e_H,a\cdot_G\huaB(e_H)\cdot_Gb)=(e_H,a\cdot_Gb).
\end{eqnarray*}
For all $(e_H,a)\in \{e_H\}\times G$, we have
\begin{eqnarray*}
(e_H,a)^{\dagger}&\stackrel{\meqref{inverse-big}}{=}&(e_H,\huaB(e_H)^{-1}\cdot_G(\huaB(e_H)\cdot_G a)^{-1})=(e_H, a^{-1}).
\end{eqnarray*}
Thus, we deduce that $\{e_H\}\times G$ is a subgroup of $(H\times G,\ast)$. For all $h\in H$ and $a\in G$, we have
\begin{eqnarray*}
(h,e_G)\ast(e_H,a)&\stackrel{\meqref{mathched-pair}}{=}&(h,\huaB(e_H)\cdot_Ga)=(h,a).
\end{eqnarray*}
It is obvious that the unit of the group $(H\times G,\ast)$ is given by $e=(e_H,e_G)$, and $H\times\{e_G\}\cap \{e_H\}\times G=\{(e_H,e_G)\}$. Therefore,  $(H\times G,\ast)$  has a group factorization into  subgroups $H\times\{e_G\}$ and $\{e_H\}\times G$.
\end{proof}

\subsection{Post-groups and the Yang-Baxter equation}
\mlabel{sec:cohomologyRB}

We now show that a post-group gives rise to a braided group, and thus leads to a solution of the Yang-Baxter equation. Thanks to the relation between relative Rota-Baxter operators and post-groups, relative Rota-Baxter operators also give rise to solutions of the Yang-Baxter equation. We further show that the construction from post-groups to braided  sets is functorial, which admits a left adjoint functor from the category of braided  sets to the category of post-groups.

\begin{defi}
Let $X$ be a set. A set-theoretical solution of the {\bf Yang-Baxter equation} on
$X$ is a bijective map $R:X\times X\lon X\times X$ satisfying:
\begin{eqnarray}
R_{12}R_{23}R_{12}=R_{23}R_{12}R_{23},\,\,\,\,\text{where}\,\,\,\,R_{12}=R\times\Id_X,~R_{23}=\Id_X\times R.
\end{eqnarray}
A set $X$ with a set-theoretical solution of
the Yang-Baxter equation on $X$ is called a {\bf braided set} and
is denoted   by $(X,R)$. Moreover, we denote by
$R(x,y)=(\varphi_{x}(y),\psi_{y}(x))$
 for all $x,y\in X$.   $R$  is called {\bf non-degenerate} if for all $x,y\in X$, the maps $\varphi_{x}$ and $\psi_{y}$ are bijective.
\end{defi}

\begin{defi}
  Let  $(X,R)$ and $(X',R')$ be braided sets. A homomorphism from $(X,R)$ to $(X',R')$ is a map $f:X\to X'$ such that
\begin{eqnarray}
(f\times f)\circ R=R'\circ(f\times f).
\end{eqnarray}
\end{defi}

It is obvious that non-degenerate braided sets and homomorphisms between non-degenerate braided sets form a category $\BS$.

\begin{defi} \mcite{Ta}
A {\bf matched pair of groups} is a triple $(G,H,\sigma)$, where $G$ and $H$ are groups and
$$\sigma:G\times H\lon H\times G,\,\,\,\,(a,h)\mapsto(a\rightharpoonup h,a\leftharpoonup h)$$
is a map satisfying the following conditions:
\begin{eqnarray}
\mlabel{MG-1}e_G\rightharpoonup h&=&h,\\
\mlabel{MG-2}a\rightharpoonup(b\rightharpoonup h)&=&(ab)\rightharpoonup h,\\
\mlabel{MG-3}(ab)\leftharpoonup h&=&\big(a\leftharpoonup(b\rightharpoonup h)\big)(b\leftharpoonup h),\\
\mlabel{MG-4}a\leftharpoonup e_H&=& a,\\
\mlabel{MG-5}(a\leftharpoonup h)\leftharpoonup k&=&a\leftharpoonup(hk),\\
\mlabel{MG-6}a\rightharpoonup(hk)&=&(a\rightharpoonup h)\big((a\leftharpoonup h)\rightharpoonup k\big),
\end{eqnarray}
for all   $a,b\in G$,~$h,k\in H.$
\end{defi}

\begin{rmk}
Let $(G,H,\sigma)$ be a matched pair of groups. Then $\sigma$ is bijective. Moreover, the triple $(H,G,\sigma^{-1})$ is also a  matched pair of groups.
\end{rmk}

Relative Rota-Baxter operators on groups naturally give rise to matched pairs of groups.

\begin{pro}\mlabel{rrbmp}
Let $\huaB:H\to G$ be a relative Rota-Baxter operator on a group $(G,\cdot_G)$ with respect to an action $\Phi$ on a group $(H,\cdot_H)$. Define $\rightharpoonup:G\times H\to H$ and $\leftharpoonup:G\times H\to G$ respectively by
\begin{eqnarray*}
a\rightharpoonup h&=&\Phi(a)(h),\\
a\leftharpoonup h &=&\huaB(\Phi(a)(h))^{-1}\cdot_Ga\cdot_G\huaB(h) \tforall h\in H,~a\in G.
\end{eqnarray*}
Then $((G,\cdot_G),(H,\circ_H),\sigma)$ is a matched pair of groups, where $\sigma(a,h)=(a \rightharpoonup h,a\leftharpoonup h)$.
\end{pro}

\begin{proof}
It follows from the fact that $(H\times G,\ast)$  is  a group factorization into  subgroups $H\times\{e_G\}$ and $\{e_H\}\times G$ by Proposition \mref{pro:factor}.
\end{proof}

\begin{defi} \mcite{Ta} \mlabel{de:bg}
A {\bf braided group} is a pair $(G,\sigma)$, where $G$ is a group and $\sigma:G\times G\lon G\times G$ is a map such that
\begin{enumerate}
\item[\rm(i)] the triple $(G,G,\sigma)$ is a matched pair of groups,
\item[\rm(ii)] $(a\rightharpoonup b)(a \leftharpoonup b)=ab \tforall a,b\in G$.
\end{enumerate}
For braided groups $(G,\sigma)$ and $(G',\sigma')$, a homomorphism of braided groups from $(G,\sigma)$ to $(G',\sigma')$ is a homomorphism of groups $f:G\to G'$ such that
\begin{eqnarray}\mlabel{BG-homo}
(f\times f)\circ \sigma=\sigma'\circ(f\times f).
\end{eqnarray}
\end{defi}

Braided groups and homomorphisms between braided groups form a category $\BG$.

\begin{thm}{\rm (\mcite{LYZ})}\mlabel{bgybe}
Let $(G,\sigma)$ be a braided group. Then $\sigma$ is a non-degenerate solution of the Yang-Baxter equation on the set $G$.
\end{thm}

Now we are ready to give the main result in this subsection, namely a post-group naturally gives rise to a
set-theoretical solution of the Yang-Baxter equation. Let
$(G,\cdot_G,\rhd_G)$ be a post-group. Define
\begin{eqnarray}
R_G:G\times G\lon G\times G, \quad R_G(a,b)=(a\rhd_G b,(a\rhd_G b)^\dagger\circ_G a\circ_G b) \tforall a,b\in G,
\end{eqnarray}
where $\circ_G$ is the sub-adjacent group structure given in Theorem \mref{pro:subad}.

\begin{pro}\mlabel{pgybe}
Let $(G,\cdot_G,\rhd_G)$ be a post-group. Then  $\big((G,\circ_G),R_G\big)$ is a braided group, and $R_G$   is a solution of the  Yang-Baxter equation  on the set $G$.
Moreover, let $\Psi$ be a  homomorphism of post-groups from  $(G,\cdot_G,\rhd_G)$ to $(H,\cdot_H,\rhd_H)$. Then $\Psi$ is a homomorphism of braided groups from $\big((G,\circ_G),R_G\big)$ to $\big((H,\circ_H),R_H\big)$.
\end{pro}

\begin{proof}
By Proposition \mref{post-group-to-RB}, the identity map $\Id:G\lon G$ is a relative Rota-Baxter operator on the sub-adjacent group  $(G,\circ_G)$ with respect to the action $L^\rhd$ on the group $(G,\cdot_G)$. Moreover, by Proposition \mref{rrbmp},  $\big((G,\circ_G),(G,\circ_G),R_G\big)$ is a matched pair of groups, where
\begin{eqnarray*}
a\rightharpoonup b&=&a\rhd_G b,\\
a\leftharpoonup b &=&(a\rhd_G b)^{\dagger}\circ_G a\circ_G b \tforall a,b\in G.
\end{eqnarray*}

Furthermore, we have $$(a\rightharpoonup b)\circ_G (a\leftharpoonup b)=(a\rhd_G b)\circ_G(a\rhd_G b)^{\dagger}\circ_G a\circ_G b=   a\circ_G b.$$ Therefore,  $\big((G,\circ_G),R_G\big)$ is a braided group, and  $R_G$ is a solution of the  Yang-Baxter equation  on the set $G$  by Theorem \mref{bgybe}.

Let $\Psi$ be a  homomorphism of post-groups from  $(G,\cdot_G,\rhd_G)$ to $(H,\cdot_H,\rhd_H)$.  By Theorem~\mref{pro:subad} \meqref{it:subad3}, $\Psi$ is a homomorphism of  groups from  $(G,\circ_G)$ to $(H,\circ_H)$.
Moreover, for all $a,b\in G$, we have
\begin{eqnarray*}
(\Psi\times \Psi)(R_G(a,b))&=&\Big(\Psi(a\rhd_G b),\Psi\big((a\rhd_G b)^\dagger\circ_G a\circ_G b\big)\Big)\\
&\stackrel{\meqref{Post-homo-1},\meqref{BG-homo-1}}{=}&\Big(\Psi(a)\rhd_H \Psi(b),\big(\Psi(a)\rhd_H \Psi(b)\big)^\dagger\circ_H \Psi(a)\circ_H \Psi(b)\Big)\\
&=&R_H\big((\Psi\times \Psi)(a,b)\big).
\end{eqnarray*}
Thus,    $\Psi$ is a homomorphism of braided groups from  $\big((G,\circ_G),R_G\big)$ to $\big((H,\circ_H),R_H\big)$.
\end{proof}

Since a relative Rota-Baxter operator naturally induces a post-group, a relative Rota-Baxter operator naturally provides a solution of the Yang-Baxter equation, giving the following conclusion.

\begin{cor}
Let $\huaB:H\longrightarrow G$ be a  relative Rota-Baxter operator on a group $(G,\cdot_G)$ with respect to an action $\Phi$. Then $R_\huaB:H\times H\to H\times H$ defined by
\begin{eqnarray}
R_\huaB(h,k)=\Big(\Phi(\huaB(h))(k),(\Phi(\huaB(h))(k))^\dagger \circ_H h\circ_H k\Big) \tforall h,k\in H,
\end{eqnarray}
is a solution of the  Yang-Baxter equation  on the set $H$, where $\circ_H$ is the descendant group structure given by \meqref{eq:descen}.
\end{cor}

\begin{rmk}\mlabel{pre-group-symmetric-group}
Let $(G,\cdot_G,\rhd_G)$ be a pre-group. Then  $\big((G,\circ_G),R_G\big)$ is a braided group with $R_G^2=\Id$.
\end{rmk}

\begin{ex}
Consider the pre-group $(\huaB_{\mathbb R},\cdot,\rhd)$ given in Corollary \mref{Butcher-G}. Then $R:\huaB_{\mathbb R}\times \huaB_{\mathbb R}\to \huaB_{\mathbb R}\times \huaB_{\mathbb R}$ defined by
$$
R(a,b)=\Big(a\rhd b, (a\rhd b)^{^\dagger}\circ a\circ b \Big),\quad a,b\in \huaB_{\mathbb R},
$$
is a solution of the Yang-Baxter equation on the set $\huaB_{\mathbb R}$. More precisely, we have
{\small \begin{eqnarray*}
(a\rhd b)(\omega)&=&\sum_{c\in AC(\omega)}a\big(P^c(\omega)\big)b\big(R^c(\omega)\big),\\
((a\rhd b)^{^\dagger}\circ a\circ b)(\omega)&=&\sum_{\omega}\sum_{c\in C(\omega_{(1)})}\sum_{\tau\in AC(\omega_{(1)}\setminus c)}(-1)^{|c|+t(\omega_{(1)})}a\big(P^\tau(\omega_{(1)}\setminus c)\big)b\big(R^\tau(\omega_{(1)}\setminus c)\big)a(\omega_{(2)})b(\omega_{(3)}),
\end{eqnarray*}
}

\noindent
where ${\Delta_{\BCK}}^{(2)}(\omega)=\sum_{\omega}\omega_{(1)}\otimes\omega_{(2)}\otimes \omega_{(3)}$ and $t(\omega_{(1)})$ is the number of trees in the rooted forest $\omega_{(1)}$.
\end{ex}

Let $\big((G,\circ_G),R_G\big)$ be a braided group. We write
$R_G(a,b)=(a\rightharpoonup b,a\leftharpoonup b)$ for all $a,b\in
G$, and define $\rhd_G:G\times G\lon G$ and $\cdot_G:G\times G\lon
G$ respectively by
\begin{eqnarray*}
a\rhd_G b=a\rightharpoonup b,\quad
a\cdot_G b=a\circ_G (a^\dagger\rightharpoonup b) \tforall a,b\in G.
\end{eqnarray*}
Here $a^\dagger$ is the inverse of $a$ in the group $(G,\circ_G).$
\begin{pro}\mlabel{ybepg}
Let $\big((G,\circ_G),R_G\big)$ be a braided group. Then $(G,\cdot_G,\rhd_G)$ is a post-group.
Moreover, let $f$ be a  homomorphism of
braided groups from $\big((G,\circ_G),R_G\big)$ to
$\big((H,\circ_H),R_H\big)$. Then $f$ is a homomorphism of
post-groups from $(G,\cdot_G,\rhd_G)$ to $(H,\cdot_H,\rhd_H)$.
\end{pro}
\begin{proof}
Since $\big((G,\circ_G),R_G\big)$ is a braided group, by \meqref{MG-1}-\meqref{MG-2} and \meqref{MG-4}-\meqref{MG-5}, we deduce that $\rightharpoonup$ and $\leftharpoonup$ are respectively the left action and right action of the group $(G,\circ_G)$ on the set $G$. Then we have
\begin{eqnarray*}
&&a\cdot_G e_G=a\circ_G (a^\dagger\rightharpoonup e_G)\stackrel{\meqref{MG-4},\meqref{MG-6}}{=}a,\\
&&a\cdot_G (a\rightharpoonup a^\dagger)=a\circ_G \big(a^\dagger\rightharpoonup (a\rightharpoonup a^\dagger)\big)\stackrel{\meqref{MG-1},\meqref{MG-2}}{=}e_G.
\end{eqnarray*}
Thus, $e_G$ is  a right unit of the multiplication $\cdot_G$ and
$a\rightharpoonup a^\dagger$ is  a  right inverse of $a\in G.$ For $a,b,c\in G$, we have
\begin{eqnarray*}
a\rhd_G (b\cdot_G c)&=&a\rightharpoonup \big(b\circ_G (b^\dagger\rightharpoonup c)\big) \stackrel{\meqref{MG-6}}{=}(a\rightharpoonup b)\circ_G \Big((a\leftharpoonup b)\rightharpoonup(b^\dagger\rightharpoonup c)\Big),\\
(a\rhd_G b)\cdot_G (a\rhd_G c)&=&(a\rightharpoonup b)\cdot_G (a\rightharpoonup c)=(a\rightharpoonup b)\circ_G\Big((a\rightharpoonup b)^\dagger\rightharpoonup (a\rightharpoonup c)\Big).
\end{eqnarray*}
Moreover, since $(a\rightharpoonup b)\circ_G(a \leftharpoonup b)=a \circ_G b$, we have
\begin{eqnarray*}
(a\rightharpoonup b)\rightharpoonup\Big((a\leftharpoonup b)\rightharpoonup(b^\dagger\rightharpoonup c)\Big)&\stackrel{\meqref{MG-2}}{=}&(a \circ_G b)\rightharpoonup(b^\dagger\rightharpoonup c)\stackrel{\meqref{MG-2}}{=}a\rightharpoonup c\\
&=&(a\rightharpoonup b)\rightharpoonup\Big((a\rightharpoonup b)^\dagger\rightharpoonup (a\rightharpoonup c)\Big).
\end{eqnarray*}
Thus, we obtain $\Big((a\leftharpoonup b)\rightharpoonup(b^\dagger\rightharpoonup c)\Big)=\Big((a\rightharpoonup b)^\dagger\rightharpoonup (a\rightharpoonup c)\Big)$. Moreover, we have
\begin{eqnarray}\mlabel{rhd-action}
a\rhd_G (b\cdot_G c)=(a\rhd_G b)\cdot_G (a\rhd_G c).
\end{eqnarray}
For all $a,b,c\in G$, we have
\begin{eqnarray*}
a\cdot_G(b\cdot_G c)&=&a\circ_G\Big(a^\dagger\rightharpoonup(b\cdot_G c)\Big)\\
&\stackrel{\meqref{rhd-action}}{=}& a\circ_G\Big((a^\dagger\rightharpoonup b)\cdot_G (a^\dagger\rightharpoonup c)\Big)\\
&=&a\circ_G\Big((a^\dagger\rightharpoonup b)\circ_G \big((a^\dagger\rightharpoonup b)^\dagger\rightharpoonup(a^\dagger\rightharpoonup c)\big)\Big)\\
&=&\big(a\circ_G(a^\dagger\rightharpoonup b)\big) \circ_G\big((a^\dagger\rightharpoonup b)^\dagger\rightharpoonup(a^\dagger\rightharpoonup c)\big)\\
&\stackrel{\meqref{MG-2}}{=}&\big(a\circ_G(a^\dagger\rightharpoonup b)\big) \circ_G\Big(\big(a\circ_G(a^\dagger\rightharpoonup b)\big)^\dagger\rightharpoonup c)\Big)\\
&=&(a\cdot_G b)\cdot_G c.
\end{eqnarray*}
Thus, we deduce that $(G,\cdot_G)$ is a group. Moreover, for all $a,b\in G$, we have
\begin{eqnarray*}
a\cdot_G(a\rhd_G b)&=&a\circ_G\big(a^\dagger\rightharpoonup(a\rightharpoonup b)\big)=a\circ_G b,\\
a\rhd_G (b\rhd_G c)&\stackrel{\meqref{MG-2}}{=}&(a\circ_G b)\rhd_G c=(a\cdot_G(a\rhd_G b))\rhd_G c,
\end{eqnarray*}
which implies  that $(G,\cdot_G,\rhd_G)$ is a post-group.

Let $f$ be  a  homomorphism of braided
groups from $\big((G,\circ_G),R_G\big)$ to $\big((H,\circ_H),R_H\big)$. Then $f:(G,\circ_G)\lon (H,\circ_H)$ is a group homomorphism such that
\begin{eqnarray*}
f(a \rightharpoonup b)=f(a)\rightharpoonup f(b),\quad
f(a \leftharpoonup b)=f(a)\leftharpoonup f(b) \tforall a,b\in G.
\end{eqnarray*}
Furthermore, for all $a,b\in G$, we have
\begin{eqnarray*}
f(a\rhd_G b)&=&f(a\rightharpoonup b)=f(a)\rightharpoonup f(b)=f(a)\rhd_H f(b),\\
f(a\cdot_G b)&=&f\big(a\circ_G (a^\dagger\rightharpoonup b)\big)=f(a)\circ_G f(a^\dagger\rightharpoonup b)=f(a)\circ_G \Big(f(a)^\dagger\rightharpoonup f(b)\Big)=f(a)\cdot_H f(b),
\end{eqnarray*}
which implies that  $f$ is a homomorphism of post-groups from $(G,\cdot_G,\rhd_G)$ to $(H,\cdot_H,\rhd_H)$.
\end{proof}

Moreover, we have
\begin{thm}\mlabel{pgybef}
Proposition \mref{pgybe} defines a functor $\huaY:\PG\lon
\BG$, and Proposition \mref{ybepg} defines a functor $\huaP:\BG\lon
  \PG$ that is the inverse of $\huaY$, giving an isomorphism between the categories $\PG$ and $\BG$.
\end{thm}

\begin{proof}
Let $(G,\cdot,\rhd)$ be a post-group. Propositions \mref{pgybe} and \mref{ybepg}
yield $(\huaP\huaY)(G,\cdot,\rhd)=(G,\cdot,\rhd)$. Let $\Psi$ be a
homomorphism of post-groups  from $(G,\cdot_G,\rhd_G)$ to
$(H,\cdot_H,\rhd_H)$. Propositions \mref{pgybe} and
\mref{ybepg} yield $(\huaP\huaY)(\Psi)=\Psi$. Thus, $\huaP\huaY=\Id$.
Similarly, $\huaY\huaP=\Id$. Therefore, the functor $\huaY$ is the
inverse of the functor $\huaP$.
\end{proof}

By Theorem \mref{bgybe},   a braided group $(G,\sigma)$ is a non-degenerate braiding set by forgetting the group structure. So there is a forgetful functor
$$\huaU:\BG\lon \BS$$
from the category of braided groups to the category of non-degenerate braided sets.

Let $(X,R)$ be a non-degenerate braided set. The structure group $G_X$ is defined to be the group generated by the elements of $X$ with the defining relations
\begin{eqnarray}
xy=uv\quad  \mbox{when} \quad R(x,y)=(u,v)\tforall x,y\in X.
\end{eqnarray}
The construction of the structure group $G_X$ was given by Etingof-Schedler-Soloviev in \mcite{ESS}. Further, Lu-Yan-Zhu \mcite{LYZ}  extended the bijective map $R$ to a braided group structure $\bar{R}$ on the structure group $G_X$.
Moreover, Lu-Yan-Zhu's construction of the braided groups is functorial, i.e. there is a structure group functor $\huaF:\BS\lon\BG$   given by
\begin{eqnarray*}
\huaF(X,R)=(G_X,\bar{R}) \tforall (X,R)\in  \BS.
\end{eqnarray*}

\begin{thm}\mlabel{th:bg}{\rm \cite[Theorem 4]{LYZ}}
With the above notations,  the functor $\huaF$ is left adjoint to $\huaU$.
\end{thm}

Moreover,  we have the functor  $\huaU\circ \huaY:\PG\lon
\BS$ from the category of post-groups to the category of non-degenerate braided sets, and the functor $\huaP\circ \huaF:\BS\lon
\PG$ from the category of non-degenerate braided sets to the category of post-groups.

\begin{thm}\mlabel{bspg}
With the above notations,  the functor $\huaP\circ \huaF$ is left adjoint to $\huaU\circ \huaY$.
\end{thm}
\begin{proof}
It follows directly from Theorems \mref{pgybef} and
 \mref{th:bg}.
\end{proof}

\subsection{Post-groups and skew-left braces}\mlabel{post-group-skew-left-brace}
We now show that there is a one-to-one correspondence between the
sets of post-groups and skew-left braces.

\begin{defi}{\rm \mcite{GV}}
A skew-left brace  $(G,\circ,\cdot)$ consists of a group $(G,\cdot)$ and a group $(G,\circ)$ such that
\vspace{-.1cm}
\begin{eqnarray}
a\circ (b\cdot c)=(a\circ b)\cdot a^{-1}\cdot (a\circ c)\tforall a,b,c\in G.
\end{eqnarray}
Here $a^{-1}$ is the inverse of $a$ in the group $(G,\cdot)$.

Let  $(G,\circ_G,\cdot_G)$ and $(H,\circ_H,\cdot_H)$ be skew-left braces. A homomorphism from $(G,\circ_G,\cdot_G)$ to $(H,\circ_H,\cdot_H)$ is a map $\Psi:G\to H$ such that
\vspace{-.1cm}
\begin{equation}
\mlabel{brace-homo-1}\Psi(a\circ_G b)=\Psi(a)\circ_H \Psi(b),\quad
\Psi(a\cdot_G b)=\Psi(a)\cdot_H \Psi(b) \tforall a,b\in G.
\end{equation}
\end{defi}

Skew-left braces and their homomorphisms form a category $\SLB$.

By Theorem \mref{pro:subad}, a post-group $(G,\cdot,\rhd)$ gives rise to a sub-adjacent group $(G,\circ)$. Moreover, it leads to a  skew-left brace.

\begin{pro}\mlabel{post-to-brace}
Let $(G,\cdot,\rhd)$ be a post-group. Then $(G,\circ,\cdot)$ is a skew-left brace.
Moreover, let $\Psi:(G,\cdot_G,\rhd_G)\to (H,\cdot_H,\rhd_H)$ be a homomorphism of post-groups. Then $\Psi$ induces a homomorphism of skew-left braces from $(G,\circ_G,\cdot_G)$ to $(H,\circ_H,\cdot_H)$.
\end{pro}

\begin{proof}
By the definition of the sub-adjacent group $(G,\circ)$, for $a,b,c\in G$,  we have
\begin{eqnarray*}
a\circ (b\cdot c)&=&a\cdot(a\rhd (b\cdot c))\stackrel{\meqref{Post-2}}{=}a\cdot((a\rhd b)\cdot (a\rhd c))\\
 &=&(a\cdot(a\rhd b))\cdot a^{-1}\cdot a\cdot (a\rhd c)=(a\circ b)\cdot a^{-1}\cdot (a\circ c),
\end{eqnarray*}
which implies that $(G,\circ,\cdot)$ is a skew-left brace.

Moreover, we have
$$
\Psi (a\circ_G b)=\Psi (a\cdot_G(a\rhd_G
b))\stackrel{\meqref{Post-homo-1}}{=}\Psi (a)\cdot_H(\Psi
(a)\rhd_H \Psi (b))=\Psi (a)\circ_H\Psi(b) \tforall a,b\in
G.
$$
Thus,   $\Psi$ is a homomorphism of
skew-left braces from $(G,\circ_G,\cdot_G)$ to $(H,\circ_H,\cdot_H)$.
\end{proof}

\begin{ex}\label{finite-left-brace}
Consider the pre-group $(\huaB_{\mathbb F_p}^n,\cdot,\rhd)$ given in Example \mref{truncation-finite-pre-groups}. Then $(\huaB_{\mathbb F_p}^n,\circ,\cdot)$ is a finite left brace. More precisely, we have
\begin{eqnarray*}
&&(a\cdot b)(\emptyset)=1,~(a\cdot b)(\omega)=a(\omega)+b(\omega),  \\
&&(a\circ b)(\emptyset)=1,~(a\circ b)(\omega)=a(\omega)+\sum_{c\in AC(\omega)}a\big(P^c(\omega)\big)b\big(R^c(\omega)\big),\tforall \omega\in \huaT^n.
\end{eqnarray*}
\end{ex}

\begin{pro}\mlabel{brace-to-post}
Let $(G,\circ,\cdot)$ be a skew-left brace. Define a multiplication $\rhd:G\times G\lon G$ by
\begin{eqnarray}
a\rhd b\coloneqq a^{-1}\cdot(a\circ b) \tforall a,b\in G,
\end{eqnarray}
where $a^{-1}$ is the inverse of $a$ in $(G,\cdot)$. Then $(G,\cdot,\rhd)$ is a post-group.
Moreover, let $\Psi$ be  a  homomorphism of skew-left braces from $(G,\circ_G,\cdot_G)$ to $(H,\circ_H,\cdot_H)$. Then $\Psi$ is a homomorphism of post-groups from $(G,\cdot_G,\rhd_G)$ to $(H,\cdot_H,\rhd_H)$.
\end{pro}

\begin{proof}
Let $(G,\circ,\cdot)$ be a skew-left brace. It is proved in \cite[Corollary 1.10]{GV} that the left multiplication $L^\rhd:G\lon \Aut(G)$ defined by $L^\rhd_ab=a\rhd b$ for all $a,b\in G$ is an action of $(G,\circ)$  on $(G,\cdot)$. Thus, we obtain that $(G,\cdot,\rhd)$ is a post-group.

For the second conclusion, for all $a,b\in G$, we have
$$
\Psi (a\rhd_G b)=\Psi (a^{-1}\cdot_G(a\circ_G b))\stackrel{\meqref{brace-homo-1}}{=}\Psi (a)^{-1}\cdot_H(\Psi (a)\circ_H \Psi (b))=\Psi (a)\rhd_H\Psi(b),
$$
showing that $\Psi$ is a homomorphism of post-groups.  \end{proof}

\begin{thm}\mlabel{functor} Proposition \mref{post-to-brace} gives a functor $\FF:\PG\lon
\SLB$, and Proposition  \mref{brace-to-post}
gives  a functor $\GG:\SLB\lon
  \PG$. Moreover, they give an isomorphism between the categories $\PG$  and $\SLB$.
\end{thm}

\begin{proof}
Let $(G,\cdot,\rhd)$ and $(H,\cdot_H,\rhd_H)$ be post-groups, and
$\Psi$ be a homomorphism of post-groups  from $(G,\cdot_G,\rhd_G)$
to $(H,\cdot_H,\rhd_H)$.  By Propositions \mref{post-to-brace} and
\mref{brace-to-post}, we have
$\GG\FF(G,\cdot,\rhd)=(G,\cdot,\rhd)$, and $(\GG\FF)(\Psi)=\Psi$.
Thus, $\GG\FF=\Id$. Similarly, $\FF\GG=\Id$. Therefore, the
functor $\FF$ is the inverse of the functor $\GG$.
\end{proof}

\section{Differentiations of post-Lie groups and post-Lie algebras}\mlabel{post-group-diff}

In this section, we take the base field to be $\mathbb R$. We show that a post-Lie group gives rise to a post-Lie algebra via differentiation, which justifies the terminology of post-groups. In addition, the relations between post-Lie groups and relative Rota-Baxter operators given by Proposition \mref{post-group-to-RB} and Theorem \mref{thm:RBPost} can also be differentiated to give relations between post-Lie algebras and relative Rota-Baxter operators on Lie algebras.

Recall~\mcite{Val} that a {\bf post-Lie algebra} $(\g,[\cdot,\cdot]_\g,\triangleright)$ consists of a Lie algebra $(\g,[\cdot,\cdot]_\g)$ and a multiplication $\triangleright:\g\otimes\g\to\g$ such that
\begin{eqnarray}
\mlabel{Posta-1}x\triangleright[y,z]_\g&=&[x\triangleright y,z]_\g+[y,x\triangleright z]_\g,\\
\mlabel{Posta-2}([x,y]_\g+x\triangleright y-y\triangleright x) \triangleright z&=&x\triangleright(y\triangleright z)-y\triangleright(x\triangleright z)\tforall x, y, z\in \g.
\end{eqnarray}
If the Lie bracket $[\cdot,\cdot]_\g$ in a post-Lie algebra $(\g,[\cdot,\cdot]_\g,\triangleright)$ is zero, then $(\g,\triangleright)$ becomes a  pre-Lie algebra. Thus,  a post-Lie algebra can be viewed as a nonabelian generalization of a pre-Lie algebra \cite{Ba,Bur-1,Ma}.

Let $(\g,[\cdot,\cdot]_\g,\triangleright_\g)$ and $(\h,[\cdot,\cdot]_\h,\triangleright_\h)$ be post-Lie algebras. A homomorphism from $(\g,[\cdot,\cdot]_\g,\triangleright_\g)$ to $(\h,[\cdot,\cdot]_\h,\triangleright_\h)$ is a Lie algebra homomorphism $\psi:\g\to \h$ such that
$$
\psi(x\triangleright_\g y)=\psi(x)\triangleright_\h
\psi(y) \tforall x,y\in \g.
$$
Post-Lie algebras together with post-Lie algebra homomorphisms form a category.

We list some properties of post-Lie algebras for application \cite{BGN}.
\begin{lem}\mlabel{lem:desLie}
  A {post-Lie algebra} $(\g,[\cdot,\cdot]_\g,\triangleright)$ gives rise to a new Lie algebra $(\g,[\cdot,\cdot]_\triangleright)$, called the {\bf sub-adjacent Lie algebra} and denoted by $\g_\triangleright$, where the Lie bracket $[\cdot,\cdot]_\triangleright$ is defined by
\begin{equation}
 [x,y]_\triangleright=[x,y]_\g+x \triangleright y- y \triangleright x\tforall x,y\in\g.
\end{equation}
Moreover, \meqref{Posta-2} is equivalent to the condition
  \begin{equation}\mlabel{eq:pLequi}
  [x,y]_\triangleright\triangleright z=x\triangleright(y\triangleright z)-y\triangleright(x\triangleright z)\tforall x,y,z\in \g.
  \end{equation}
  Consequently, the map $L^\triangleright:\g\to \Der(\g)$ defined by $L^\triangleright_xy=x  \triangleright y$ gives rise to an action of the sub-adjacent Lie algebra $\g_\triangleright$ on the Lie algebra $(\g,[\cdot,\cdot]_\g)$.
\end{lem}

\begin{defi} A post-group $(G,\cdot,\rhd)$ is called a {\bf post-Lie group} if $(G,\cdot)$ is a Lie group and $\rhd$ is a smooth map. 
\end{defi}	

Let $(\g,[\cdot,\cdot]_\g)$ be the Lie algebra of the Lie group $(G,\cdot)$. Denote by $\Aut(\g)$ and $\Der(\g)$ the Lie group of automorphisms and the Lie algebra of derivations on the Lie algebra $(\g,[\cdot,\cdot]_\g)$ respectively.  Let $\exp:\g\to G$ denote the exponential map. Then the relation between the Lie bracket $[\cdot,\cdot]_\g$ and the Lie group multiplication is given by the following fundamental formula:
\begin{equation}\mlabel{eq:expo}
  [u,v]_\g=\frac{d^2}{dt ds}\,\bigg|_{t,s=0}\exp(tu)\exp(sv)\exp(-tu)\tforall u,v\in\g.
\end{equation}

Since $L^\rhd_a\in\Aut(G)$, it follows that $(L^\rhd_a)_{*e}\in \Aut(\g)$. Thus we obtain a map, still denoted by $L^\rhd$, from $G$ to $\Aut(\g)$. Then taking the differentiation, we obtain a map $L^\rhd_{*e}: \g\to \Der(\g)$. The above process can be summarized by the following diagram:
\begin{equation}\mlabel{eq:relation}
\begin{split}
\small{
 \xymatrix{G \ar[rr]^{L^\rhd}\ar[d]_{\text{differentiation}} &  &  \Aut(\g) \ar[d]^{\text{differentiation}}  \\
 \g \ar[rr]^{ L^\rhd_{*e} } & &\Der(\g).}
}
\end{split}
\end{equation}
Define $\triangleright:\g\otimes \g\to \g$ by 
\begin{equation}\mlabel{eq:diff}
  x\triangleright y= L^\rhd_{*e} (x)(y)=\frac{d}{dt}\bigg|_{t=0} L^\rhd_{\exp(tx)}y=\frac{d}{dt}\bigg|_{t=0}\frac{d}{ds}\bigg|_{s=0}L^\rhd_{\exp(tx)}\exp(sy).
\end{equation}

\begin{thm}\mlabel{thm:diffpL}
Let $(G,\cdot,\rhd)$ be a post-Lie group with the smooth multiplication $\rhd$. Then  $(\g,[\cdot,\cdot]_\g,\triangleright)$ is a post-Lie algebra.
Let $\Psi$ be a homomorphism of post-Lie groups  from $(G,\cdot_G,\rhd_G)$ to $(H,\cdot_H,\rhd_H)$. Then $\psi:=\Psi_{*e_G}$ is a homomorphism of the differentiated post-Lie algebras from $(\g,[\cdot,\cdot]_\g,\triangleright_\g)$ to $(\h,[\cdot,\cdot]_\h,\triangleright_\h)$.
Furthermore, the correspondences give a functor from the category of post-Lie groups to the category of post-Lie algebras.
\end{thm}
\begin{proof}
Since $L^\rhd_{*e} (x)\in\Der(\g)$ for all $x\in\g$, it is obvious that \meqref{Posta-1} is satisfied.

Let $x,y,z\in \g$. By \meqref{eq:diff},
\meqref{Post-4} and the BCH formula, we obtain
  \begin{eqnarray*}
    &&x\triangleright(y\triangleright z)-y\triangleright(x\triangleright z)\\
    &=&\frac{d}{dt}\bigg|_{t=0}\frac{d}{ds}\bigg|_{s=0}\frac{d}{dr}\bigg|_{r=0}\Big(L^\rhd_{\exp(tx)}L^\rhd_{\exp(sy)}\exp(rz)-L^\rhd_{\exp(sy)}L^\rhd_{\exp(tx)}\exp(rz)\Big)\\
    &=&\frac{d}{dt}\bigg|_{t=0}\frac{d}{ds}\bigg|_{s=0}\frac{d}{dr}\bigg|_{r=0}\Big(L^\rhd_{\exp(tx)\cdot(\exp(tx)\rhd \exp(sy))}\exp(rz)-L^\rhd_{\exp(sy)\cdot(\exp(sy)\rhd\exp(tx))}\exp(rz)\Big)\\
        &=&\frac{d}{dt}\bigg|_{t=0}\frac{d}{ds}\bigg|_{s=0}\frac{d}{dr}\bigg|_{r=0}L^\rhd_{\exp(tx)\cdot   \exp(sy) }\exp(rz)+\frac{d}{dt}\bigg|_{t=0}\frac{d}{ds}\bigg|_{s=0}\frac{d}{dr}\bigg|_{r=0}L^\rhd_{ \exp(tx)\rhd \exp(sy) }\exp(rz)\\
    &&-\frac{d}{dt}\bigg|_{t=0}\frac{d}{ds}\bigg|_{s=0}\frac{d}{dr}\bigg|_{r=0}L^\rhd_{\exp(sy)\cdot \exp(tx) }\exp(rz)-\frac{d}{dt}\bigg|_{t=0}\frac{d}{ds}\bigg|_{s=0}\frac{d}{dr}\bigg|_{r=0}L^\rhd_{ \exp(sy)\rhd\exp(tx) }\exp(rz)\\
      &=&\frac{d}{dt}\bigg|_{t=0}\frac{d}{ds}\bigg|_{s=0}\frac{d}{dr}\bigg|_{r=0}L^\rhd_{\exp(tx+sy+\frac{1}{2}ts[x,y]_\g+\cdots) }\exp(rz)+ (x\triangleright y)\triangleright z\\
    &&-\frac{d}{dt}\bigg|_{t=0}\frac{d}{ds}\bigg|_{s=0}\frac{d}{dr}\bigg|_{r=0}L^\rhd_{\exp(sy+tx+\frac{1}{2}ts[y,x]_\g+\cdots) }\exp(rz)- (y\triangleright x)\triangleright z\\
      &=&\frac{1}{2}[x,y]_\g\triangleright z+ (x\triangleright y)\triangleright z-\frac{1}{2}[ y,x]_\g\triangleright z- (y\triangleright x)\triangleright z\\
      &=&[x,y]_ \triangleright \triangleright z,
  \end{eqnarray*}
  which implies that \meqref{eq:pLequi} holds. Thus by Lemma \mref{lem:desLie}, $(\g,[\cdot,\cdot]_\g,\triangleright)$ is a post-Lie algebra.

  Since $\Psi$ is a Lie group homomorphism, it follows that $\psi$ is a Lie algebra homomorphism. Moreover, we have
  \begin{eqnarray*}
    \psi(x)\triangleright_\h\psi( y)&=&\frac{d}{dt}\bigg|_{t=0}\frac{d}{ds}\bigg|_{s=0}L^{\rhd_H}_{\exp(t\psi(x))} \exp(s\psi( y))=\frac{d}{dt}\bigg|_{t=0}\frac{d}{ds}\bigg|_{s=0}L^{\rhd_H}_{\Psi(\exp(t x ))} \Psi(\exp(s y))\\
    &=&\frac{d}{dt}\bigg|_{t=0}\frac{d}{ds}\bigg|_{s=0}\Psi(L^{\rhd_G}_{\exp(t x )}  \exp(s y))=\psi(\frac{d}{dt}\bigg|_{t=0}\frac{d}{ds}\bigg|_{s=0} L^{\rhd_G}_{\exp(t x )}  \exp(s y))\\
    &=&\psi(x\triangleright_\g y ),
  \end{eqnarray*}
  which implies that $\psi$ is a homomorphism between post-Lie algebras.

It is straightforward to check that the correspondences indeed define a functor.
\end{proof}

\begin{rmk}
Rump constructed pre-Lie algebras from $\mathbb R$-braces in \cite{Ru14}. This can be regarded as a special case of Theorem~\ref{thm:diffpL} thanks to Theorem \ref{functor} for pre-Lie groups.
\end{rmk}

\begin{rmk}
There are similar constructions of pre-Lie rings from finite braces \cite{SS,Sm22a,Sm22b,Sm22c,Sm22d}. Since braces are special skew-left braces, it is natural to ask for a ``differentiation'' theory for finite nilpotent post-groups.
\end{rmk}
In the sequel, we show that the differentiation of the sub-adjacent  Lie group $(G,\circ)$ given in Theorem \mref{pro:subad} is exactly the sub-adjacent Lie algebra $(\g,[\cdot,\cdot]_\triangleright)$ given in Lemma \mref{lem:desLie}.

\begin{lem}\mlabel{lem:diffi}Let $(G,\cdot,\rhd)$ be a post-Lie group with the smooth multiplication $\rhd$. Then we have
  $$
  \frac{d}{dt}\bigg|_{t=0}\exp(tx)^{\dagger}=-x\tforall x\in\g.
  $$
\end{lem}
\begin{proof}
  By the Leibniz rule, we have
  $$
 0= \frac{d}{dt}\bigg|_{t=0}(\exp(tx)^{\dagger}\circ \exp(tx))=  \frac{d}{dt}\bigg|_{t=0}\exp(tx)^{\dagger}+x,
  $$
  which implies that $
  \frac{d}{dt}\big|_{t=0}\exp(tx)^{\dagger}=-x.
  $
\end{proof}

\begin{pro}Let $(G,\cdot,\rhd)$ be a post-Lie group with the smooth multiplication $\rhd$. Then the Lie algebra of the sub-adjacent Lie group $(G,\circ)$ is the sub-adjacent Lie algebra $(\g, [\cdot,\cdot]_\triangleright)$, giving the following commutative diagram
\begin{equation}\mlabel{eq:post-sub}
\begin{split}
 \xymatrix{\text{post-Lie group}~ (G,\cdot,\rhd) \ar[rr]^{\text{\qquad\qquad sub-adjacent}}\ar[d]_{\text{differentiation}} &  &   (G,\circ) \ar[d]^{\text{differentiation}} \\
\text{post-Lie algebra}~(\g,[\cdot,\cdot]_\g,\triangleright) \ar[rr]^{ \text{\qquad\qquad sub-adjacent}} & & (\g, [\cdot,\cdot]_\triangleright).}
\end{split}
\end{equation}
   \end{pro}

   \begin{proof} First for all $a,b\in G$, we have
\begin{eqnarray*}
  a\circ b\circ a^{\dagger}= a\circ (b\circ a^{\dagger})=a\circ (b\cdot (b\rhd a^{\dagger}))=a\cdot( a\rhd (b\cdot (b\rhd a^{\dagger})))=a\cdot(a\rhd b)\cdot ( a \rhd (b\rhd a^{\dagger})).
\end{eqnarray*}
Denote by $\{\cdot,\cdot\}$ the Lie algebra structure of the Lie group $(G,\circ)$. Then by Theorem \mref{pro:subad} and Lemma \mref{lem:diffi}, we have
\begin{eqnarray*}
 \{x,y\} &=&\frac{d^2}{dt ds}\,\bigg|_{t,s=0} \exp(tx)\circ \exp(sy)\circ \exp(tx)^\dagger\\
 &=&\frac{d^2}{dt ds}\,\bigg|_{t,s=0}\exp(tx)\cdot(\exp(tx)\rhd \exp(sy))\cdot ( \exp(tx) \rhd (\exp(sy)\rhd \exp(tx)^{\dagger}))\\
 &=&\frac{d^2}{dt ds}\,\bigg|_{t,s=0}\exp(tx)\cdot(\exp(tx)\rhd \exp(sy))\cdot ( \exp(tx) \rhd  \exp(tx)^{\dagger})\\
 &&+\frac{d^2}{dt ds}\,\bigg|_{t,s=0}\exp(tx)\cdot ( \exp(tx) \rhd (\exp(sy)\rhd \exp(tx)^{\dagger}))\\
 &=&\frac{d^2}{dt ds}\,\bigg|_{t,s=0}\exp(tx)\cdot(\exp(tx)\rhd \exp(sy))\cdot    \exp(tx)^{-1}\\
 &&+\frac{d^2}{dt ds}\,\bigg|_{t,s=0} \exp(sy)\rhd \exp(tx)^{\dagger})\\
 &=& \frac{d^2}{dt ds}\,\bigg|_{t,s=0} \exp(tx) \cdot \exp(sy) \cdot     \exp(tx)^{-1} +\frac{d^2}{dt ds}\,\bigg|_{t,s=0} \exp(tx)\rhd \exp(sy)\\
 &&+\frac{d^2}{dt ds}\,\bigg|_{t,s=0}\exp(sy)\rhd \exp(tx)^{\dagger}\\
 &=&[x,y]_\g+x\triangleright y-y\triangleright x,
\end{eqnarray*}
which implies $\{x,y\} = [x,y]_\triangleright.$
\end{proof}

Recall that given an action $\phi:\g\to \Der(\h)$ of a Lie algebra $(\g,[\cdot,\cdot]_\g)$ on a Lie algebra  $(\h,[\cdot,\cdot]_\h)$, a relative Rota-Baxter operator of weight $1$ on $\g$ with respect to the action $\phi$ is a linear map $B:\h\to \g$ satisfying
\begin{equation} \mlabel{eq:B}
[B(u), B(v)]_\g=B\Big(\phi(B(u))v-\phi(B(v))u+[u,v]_\h\Big)\tforall u, v\in\h.
\end{equation}

Let $G$ and $H$ be Lie groups whose Lie algebras are $(\g,[\cdot,\cdot]_\g)$ and $(\h,[\cdot,\cdot]_\h)$ respectively. Then taking the differentiation we can obtain the following result.

\begin{thm}{\rm(\mcite{GLS})}\mlabel{thm:diffRB}
 Let $\huaB:H\longrightarrow G$ be a   relative Rota-Baxter operator  on a Lie group $(G,\cdot_G)$ with respect to the action $\Phi$. Then $B:=\huaB_{*e_H}:\h\to \g$ is a relative Rota-Baxter operator of weight $1$ on the Lie algebra $\g$ with respect to the action $\phi:=\Phi_*.$
\end{thm}

Proposition \mref{post-group-to-RB}, and Theorems \mref{thm:diffpL} and \mref{thm:diffRB} yield the commutative diagram
\begin{equation}\mlabel{eq:diff-id}
\begin{split}
\small{
 \xymatrix{\text{post-Lie group}~ (G,\cdot,\rhd) \ar[rr]^{\text{}}\ar[d]_{\text{differentiation}} &  & \text{relative RB operator} ~\Id:G \to G_\rhd \ar[d]^{\text{differentiation}} \\
\text{post-Lie algebra}~(\g,[\cdot,\cdot]_\g,\triangleright) \ar[rr]^{ \text{}} & &\text{relative RB operator}~ \Id:\g\to\g_\triangleright.}
}
\end{split}
\end{equation}

Let $B:\h\to \g$ be a relative Rota-Baxter operator on a Lie algebra $\g$ with respect to an action $\phi:\g\to \Der(\h)$. Then $(\h,[\cdot,\cdot]_\h,\triangleright)$ is a post-Lie algebra, where $\triangleright:\h\otimes \h\to\h$
is defined by
$$
u \triangleright v=\phi(u)(v)\tforall u,v\in\h.
$$

Similarly, Theorems \mref{thm:RBPost}, \mref{thm:diffpL} and~\mref{thm:diffRB} yield the
commutative diagram
\begin{equation}\mlabel{eq:diff-act}
\begin{split}
\small{
 \xymatrix{\text{relative RB operator}~ H\stackrel{\huaB}{\to} G \ar[rr]^{\text{splitting}}\ar[d]_{\text{differentiation}} &  & \text{post-Lie group}~ (H,\cdot_H,\rhd) \ar[d]^{\text{differentiation}}  \\
\text{relative RB operator}~\h \stackrel{B}{\to} \g \ar[rr]^{ \text{splitting}} & &\text{post-Lie algebra}~(\h,[\cdot,\cdot]_\h,\triangleright).}
}
\end{split}
\end{equation}
In summary, we have the commutative square in the diagram in~\meqref{eq:bigdiag}.
\vspace{-.1cm}
\section{Formal integrations of post-Lie algebras and Lie-Butcher groups}\mlabel{post-group-inte}

In this section, we give the formal integration of connected complete post-Lie algebras using connected complete post-Hopf algebras and post-Lie Magnus expansions. As a byproduct, we show that there is a post-group underlying the Lie-Butcher group.
\vspace{-.1cm}

\subsection{Post-Hopf algebras, post-Lie algebras and post-groups}

In this subsection, first  we recall
the notion of   post-Hopf algebras   introduced in \mcite{LST} as the underlying algebraic structures of   relative Rota-Baxter operators on Hopf algebras, and the relation to post-Lie algebras. Then we show that a  post-Hopf algebra   naturally gives rise to a post-group.

\begin{defi}\mlabel{defi:pH}{\rm \mcite{LST}}
A {\bf post-Hopf algebra} is a pair $(H,\rhd)$, where $H:=(H,\cdot,1,\Delta,\vep,S)$ is a cocommutative Hopf algebra and $\rhd:H\otimes H\to H$ is a coalgebra homomorphism such that,
\begin{enumerate}
\item the following equalities hold:
\begin{eqnarray}
\mlabel{Post-hopf-2}x\rhd (y\cdot z)&=&\sum_{x}(x_{(1)}\rhd y)\cdot(x_{(2)}\rhd z),\\
\mlabel{Post-hopf-4}x\rhd (y\rhd z)&=&\sum_{x}\big(x_{(1)}\cdot(x_{(2)}\rhd y)\big)\rhd z\tforall x, y, z\in H,\quad\Delta(x)=\sum_x x_{(1)}\otimes x_{(2)};
\vspace{-.1cm}
\end{eqnarray}
\item the left multiplication $L^\rhd:H\to \End(H)$ defined by
$L^\rhd_x y= x\rhd y$ is convolution invertible in
$\Hom(H,\End(H))$, that is, there exists a unique
$\beta:H\to\End(H)$ such that
\begin{equation}\mlabel{Post-con}
  \sum_{x}L^\rhd_{x_{(1)}}\circ\beta_{x_{(2)}}=  \sum_{x}\beta_{x_{(1)}}\circ L^\rhd_{x_{(2)}}=\varepsilon(x)\Id_H\tforall x\in H,\quad\Delta(x)=\sum_x x_{(1)}\otimes x_{(2)}.
\end{equation}
\end{enumerate}
\end{defi}

\begin{rmk}
A related structure is a $D$-algebra, which plays an important role in the study of numerical Lie group integrators \mcite{MW,ML}. Recently, the notion of a $D$-bialgebra  was
introduced  in \mcite{MQS} as the universal enveloping algebra of a post-Lie algebra.
\end{rmk}
\vspace{-.2cm}

\begin{thm} \mcite{LST} \mlabel{thm:subHopf}
Let $(H,\rhd)$ be a  post-Hopf algebra. Define
\begin{equation}
    \mlabel{post-rbb-1}
x *_\rhd y=\sum_{x}x_{(1)}\cdot (x_{(2)}\rhd y),~
S_\rhd(x)=\sum_{x}\beta_{\rhd,x_{(1)}}(S(x_{(2)})) \tforall x,y\in H,~\Delta(x)=\sum_x x_{(1)}\otimes x_{(2)}.
\end{equation}
Then
$H_\rhd\coloneqq(H,*_\rhd,1,\Delta,\vep,S_\rhd)$
is a Hopf algebra, called the {\bf sub-adjacent Hopf algebra} of $(H,\rhd)$.
\end{thm}

Let $(\g,[\cdot,\cdot]_\g,\rhd)$ be a post-Lie algebra. Foissy \mcite{Fo} extended the post-Lie product $\rhd:\g \otimes \g\to \g$,
to the coshuffle Hopf algebra $(T\g,\cdot,\Delta^{\co})$ as follows: for $x,x_1,\dots,x_n,a,a_1,\dots,a_m\in \g$ and $X\in T\g,~{\Delta^{\co}}^{(m-1)}X=\sum_{X}X_{(1)}\otimes\cdots\otimes X_{(m)}$, define
\begin{eqnarray}\mlabel{post-brace-algebra-1}
\mlabel{u-1}1\rhd a&=&a,\\
\mlabel{u-2}x\rhd a&=&x\rhd a, \\
\mlabel{u-3}(x\otimes x_1) \rhd a&=&x\rhd (x_1\rhd a)-(x\rhd x_1)\rhd a,\\
\vspace{-.1cm}
\nonumber&\vdots&\\
\vspace{-.1cm}
\mlabel{u-4}(x\otimes x_1\otimes\cdots \otimes x_n) \rhd a&=&x\rhd \big((x_1\otimes\cdots \otimes x_n)\rhd a\big)\\
\nonumber&&-\sum_{i=1}^{n}\big(x_1\otimes\cdots \otimes x_{i-1}\otimes(x\rhd x_i) \otimes x_{i+1}\otimes\cdots\otimes x_n\big)\rhd a,
\vspace{-.5cm}
\end{eqnarray}
\vspace{-.1cm}
and
\vspace{-.2cm}
\begin{eqnarray}\mlabel{post-brace-algebra-2}
\mlabel{u-5}1\rhd 1&=&1,\\
\mlabel{u-6}x\rhd 1&=&0,\\
\mlabel{u-7}X \rhd (a_1\otimes\cdots\otimes a_m)&=&\sum_{X}(X_{(1)}\rhd a_1)\otimes\cdots\otimes(X_{(m)}\rhd a_m).
\vspace{-.1cm}
\end{eqnarray}
Then $(T\g,\cdot,\Delta^{\co},\rhd)$ is a post-Hopf algebra \mcite{LST}.   Moreover, Foissy \mcite{Fo} and Mencattini-Quesney-Silva \mcite{MQS}   proved that, for the ideal $J$ of $T\g$ generated by   $\big\{x\otimes y-y\otimes x-[x,y]_\g\,| \tforall x,y\in\g\big\},$ one has
$
J\rhd T\g=0$ and $T\g\rhd J\subset J.$
Therefore, the universal enveloping algebra $(U(\g),\cdot,\Delta^{\co},\rhd)$ is  also a post-Hopf algebra.
See \mcite{ELM,LST,MQS,OG} for more studies on the universal enveloping algebras of  post-Lie algebras.
\vspace{-.1cm}

\begin{thm}\mlabel{group-like}
Let $(G,\cdot,\rhd)$ be a post-group. Then the group algebra $\bk[G]$ is a post-Hopf algebra. Conversely,
let $(H,\rhd)$ be a  post-Hopf algebra.   Then the set $G(H)$ of group-like elements is a post-group.
\end{thm}
\vspace{-.5cm}

\begin{proof}
The first result is straightforward. For the second result, since $\rhd:H\otimes H\to H$ is a coalgebra homomorphism, for $x,y\in G(H)$, we have
$
\Delta(x\rhd y)=(x\rhd y)\otimes (x\rhd y).
$
Thus, $G(H)$ is closed under $\rhd$. Further, \meqref{Post-hopf-2} and \meqref{Post-hopf-4} imply \meqref{Post-2} and \meqref{Post-4} respectively. Finally \meqref{Post-con} gives
\begin{equation*}
L^\rhd_{x}\circ\beta_{x}=\beta_{x}\circ L^\rhd_{x}=\varepsilon(x)\Id_H=\Id_H\tforall x\in G(H).
\end{equation*}
Therefore, $(G(H),\cdot,\rhd)$ is a post-group.
\end{proof}
\vspace{-.4cm}

\subsection{Formal integration of connected complete post-Lie algebras}

In this subsection, we give the formal integration of
connected complete  post-Lie algebras in terms of post-groups.

\begin{defi}\mcite{Fr}
A {\bf   filtered vector space} is a pair $(V,\huaF_{\bullet}V)$, where $V$ is a vector space and $\huaF_{\bullet}V$ is a descending
filtration of the  vector space $V$ such that $V=\huaF_0V\supset\huaF_1V\supset\cdots\supset\huaF_n V\supset\cdots$.

Let $(V,\huaF_{\bullet}V)$ and $(W,\huaF_{\bullet}W)$ be filtered  vector spaces. A {\bf homomorphism} $f:(V,\huaF_{\bullet}V)\to (W,\huaF_{\bullet}W)$ is a linear map $f:V\to W$ such that $f(\huaF_n V)\subset \huaF_n W,~n\geq 0.$
\end{defi}
\vspace{-.1cm}

Denote by  $\fVec$ the category of filtered vector spaces.
For filtered vector spaces $(V,\huaF_{\bullet}V)$ and $(W,\huaF_{\bullet}W)$, there is a filtration on $V\otimes W$   given by
\begin{eqnarray}\mlabel{FTP}
\huaF_n(V\otimes W)\coloneqq \sum_{i+j=n}\huaF_iV\otimes\huaF_jW\subset V\otimes W.
\vspace{-.2cm}
\end{eqnarray}
\vspace{-.1cm}
Thus, $(V\otimes W,\huaF_{\bullet}(V\otimes W))$ is a filtered vector space, called the  {\bf tensor product} of filtered vector spaces  $(V,\huaF_{\bullet}V)$ and $(W,\huaF_{\bullet}W)$. For the ground field $\bk$, we define $\huaF_0\bk=\bk,\huaF_n\bk=\{0\},~n\geq 2$. It follows that $(\fVec,\otimes,\bk)$ is a $\bk$-linear  symmetric monoidal category.

\begin{defi}
A {\bf   filtered post-Lie algebra} $(\g,\huaF_{\bullet}\g)$  is a post-Lie algebraic object in the category $(\fVec,\otimes,\bk)$. A filtered post-Lie algebra $(\g,\huaF_{\bullet}\g)$ is called {\bf connected} when we have $\huaF_1\g=\g.$
\end{defi}

\begin{ex}\label{free-post-Lie}
Let $\huaO$ be the set of isomorphism classes of planar rooted trees:
\[
    \huaO= \Big\{\begin{array}{c}
        \scalebox{0.6}{\ab}, \scalebox{0.6}{\aabb},
        \scalebox{0.6}{\aababb}, \scalebox{0.6}{\aaabbb},\scalebox{0.6}{\aabababb},
        \scalebox{0.6}{\aaabbabb},
        \scalebox{0.6}{\aabaabbb}, \scalebox{0.6}{\aaababbb}, \scalebox{0.6}{\aaaabbbb},\scalebox{0.6}{\aababababb},\scalebox{0.6}{\aaabbababb},\scalebox{0.6}{\aabaabbabb},\scalebox{0.6}{\aababaabbb},\scalebox{0.6}{\aaababbabb},\scalebox{0.6}{\aabaababbb},\scalebox{0.6}{\aaabbaabbb},\scalebox{0.6}{\aaabababbb},\scalebox{0.6}{\aaaabbbabb},\scalebox{0.6}{\aabaaabbbb},\scalebox{0.6}{\aaaababbbb},\scalebox{0.6}{\aaaabbabbb},\scalebox{0.6}{\aaabaabbbb},\scalebox{0.6}{\aaaaabbbbb},\ldots
            \end{array}
            \Big\}.
\]
Let $\bk\{\huaO\}$ be the free $\bk$-vector space generated by  $\huaO$. The {\bf left grafting operator} $\rhd:\bk\{\huaO\}\otimes \bk\{\huaO\}\to \bk\{\huaO\}$ is defined by

\begin{eqnarray}\mlabel{free-post-product}
\tau\rhd \omega=\sum_{s\in {\rm Nodes}(\omega)}\tau\circ_{s}\omega\tforall \tau,\omega\in \huaO,
\vspace{-.1cm}
\end{eqnarray}
where $\tau\circ_{s}\omega$ is the planar  rooted tree resulting from attaching the root of $\tau$ to  the node $s$ of the tree $\omega$ from the left. Consider the free Lie algebra $Lie(\bk\{\huaO\})$ generated by the vector space $\bk\{\huaO\}$, and extend the left grafting operator $\rhd$ on $\bk\{\huaO\}$ to the free Lie algebra $Lie(\bk\{\huaO\})$ by \meqref{Posta-1} and \meqref{Posta-2}. Therefore, $(Lie(\bk\{\huaO\}),[\cdot,\cdot],\rhd)$ is a post-Lie algebra, which is the free post-Lie algebra \mcite{ML,Val} generated by  one generator $\{\scalebox{0.6}{\ab}\}$. Moreover, $(Lie(\bk\{\huaO\}),[\cdot,\cdot],\rhd)$ is a weight graded post-Lie algebra, here the weight grading is induced by the number of nodes in the trees. We denote the homogeneous component of weight grading degree $n$ by $\big(Lie(\bk\{\huaO\})\big)_n$. Then we have $Lie(\bk\{\huaO\})=\oplus_{n=1}^{+\infty}\big(Lie(\bk\{\huaO\})\big)_n$. Define a filtration $\huaF_{\bullet}$ on $Lie(\bk\{\huaO\})$ by
\[
\huaF_n\big(Lie(\bk\{\huaO\})\big):=\left\{
\begin{array}{ll}
\oplus_{n=1}^{+\infty}\big(Lie(\bk\{\huaO\})\big)_n, &n=0, \\
\oplus_{k=n}^{+\infty}\big(Lie(\bk\{\huaO\})\big)_k, &n\geq 1.
\end{array}
\right.
\]
Then $\Big((Lie(\bk\{\huaO\}),[\cdot,\cdot],\rhd),\huaF_{\bullet}\Big)$ is a connected filtered post-Lie algebra.
\end{ex}

\begin{ex}\label{finite-post-Lie}
Let $p$ be an odd prime number. Consider the free post-Lie algebra $Lie(\mathbb F_p\{\huaO\})$ on one generator $\{\scalebox{0.6}{\ab}\}$ over the finite  field $\mathbb F_p$. Then similar to the discussion in the above example,
 $\Big((Lie(\mathbb F_p\{\huaO\}),[\cdot,\cdot],\rhd),\huaF_{\bullet}\Big)$ is a connected filtered post-Lie algebra.  Moreover, for any $n\geq 1$, $\huaF_n\big(Lie(\mathbb F_p\{\huaO\})\big)$ is an ideal of the post-Lie algebra $Lie(\mathbb F_p\{\huaO\})$. Thus, we obtain
$$
Lie(\mathbb F_p\{\huaO\})/\huaF_n\big(Lie(\mathbb F_p\{\huaO\})\big)\cong\oplus_{k=1}^{n-1}\big(Lie(\mathbb F_p\{\huaO\})\big)_k
$$
is a finite nilpotent post-Lie algebra over the finite  field $\mathbb F_p$.
\end{ex}

\begin{defi}
A {\bf   filtered post-Hopf algebra} $(H,\rhd,\huaF_{\bullet} H)$  is a post-Hopf  algebraic object in the category $(\fVec,\otimes,\bk)$. A filtered Hopf algebra $(H,\rhd,\huaF_{\bullet} H)$ is called {\bf connected} if $H/\huaF_1 H=\bk.$
\end{defi}
\vspace{-.1cm}

\begin{ex}
Let $(\g,\huaF_{\bullet}\g)$ be a filtered post-Lie algebra. By \meqref{u-1}-\meqref{u-7} and \meqref{FTP},  we deduce that $\Big(\big(T\g,\cdot,\Delta^{\co},\rhd\big),\huaF_{\bullet}(T\g)\Big)$ is a filtered post-Hopf algebra. Consequently,  $\Big(\big(U(\g),\cdot,\Delta^{\co},\rhd\big),\huaF_{\bullet}U(\g)\Big)$ is also a filtered post-Hopf algebra, where the filtration is given by
\begin{eqnarray}
\huaF_{n}U(\g)=(\huaF_{n}(T\g)+J)/J \tforall n\geq 0.
\end{eqnarray}
\end{ex}

\begin{defi} \mcite{Fr}
A {\bf   complete vector space} is a filtered vector space $(V,\huaF_{\bullet}V)$ such that the natural homomorphism
\vspace{-.2cm}
\begin{eqnarray}\mlabel{complete-mor}
\Phi_V:V\to \underleftarrow{\lim}V/\huaF_n V
\end{eqnarray}
is a linear isomorphism of vector spaces. Denote by  $\cVec$ the category of complete vector spaces.
\end{defi}

Let $(V,\huaF_{\bullet}V)$ be a filtered vector space. Then the vector space $\widehat{V}=\underleftarrow{\lim}V/\huaF_n (V)$ is a complete vector space with the filtration given by
\begin{eqnarray}
\huaF_n(\widehat{V})=\ker \pi_n,\,\,\pi_n:\underleftarrow{\lim}V/\huaF_n (V)\lon V/\huaF_n (V).
\end{eqnarray}
We call the complete vector space $\big(\widehat{V},\huaF_\bullet\widehat{V}\big)$ the {\bf completion} of the filtered vector space $(V,\huaF_{\bullet}V)$.
Let $(V,\huaF_{\bullet}V)$ and $(W,\huaF_{\bullet}W)$ be complete vector spaces. The completion of the filtered vector space $(V\otimes W,\huaF_{\bullet}(V\otimes W))$ is called the {\bf complete tensor product} of $(V,\huaF_{\bullet}V)$ and $(W,\huaF_{\bullet}W)$.
We denote the complete tensor product of $(V,\huaF_{\bullet}V)$ and $(W,\huaF_{\bullet}W)$ by $(V\hat{\otimes}W,\huaF_{\bullet}(V\hat{\otimes}W))$. Moreover,  $(\cVec,\hat{\otimes},\bk)$ is a $\bk$-linear symmetric monoidal category, and the completion functor $\widehat{(-)}$ is a  symmetric monoidal functor from the  category $\fVec$ to the  category $\cVec$. See \cite[Section 7.3]{Fr}  for more details about the complete tensor product of complete vector spaces.

\begin{defi}
A {\bf   complete  post-Lie algebra} $(\g,\huaF_{\bullet}\g)$  is a post-Lie algebraic object in the category $(\cVec,\hat{\otimes},\bk)$. A complete post-Lie algebra $(\g,\huaF_{\bullet}\g)$ is called {\bf connected} if $\huaF_1\g=\g.$
\end{defi}

We easily have an explicit description of a connected complete post-Lie algebra.

\begin{pro}\label{pronilpoten-post-lie}
A connected complete post-Lie algebra is equivalent to a filtered vector space $(\g,\huaF_{\bullet}\g)$, where $\g$ is a post-Lie algebra, such that
\begin{itemize}
	\item[\rm(i)] $\g=\huaF_0\g=\huaF_1\g$; 
	\item[\rm(ii)]
	for all $m,n\ge 0$, we have
	$	[\huaF_{m}\g,\huaF_{n}\g]\subset \huaF_{m+n}\g,~~~~\huaF_{m}\g\rhd\huaF_{n}\g\subset \huaF_{m+n}\g$;
	\item[\rm(iii)]
	$\g$ is complete with respect to this filtration, i.e. there is an isomorphism of vector spaces
	$
	\g\cong\underleftarrow{\lim}~\g/\huaF_n\g
	$.
\end{itemize}
\end{pro}

We introduce the notion of a nilpotent post-Lie algebra and illustrate the fact that a nilpotent post-Lie algebra is naturally a connected complete  post-Lie algebra.

Let $(\g,[\cdot,\cdot],\rhd)$ be a post-Lie algebra. Define
\begin{eqnarray*}
\g^{[0]}=\g^{[1]}=\g,\,\,\,\,\g^{[2]}=[\g,\g]+\g\rhd\g,\,\,\,\,\g^{[n]}=\sum_{k=1}^{n-1}[\g^{[k]},\g^{[n-k]}]+\sum_{k=1}^{n-1}\g^{[k]}\rhd\g^{[n-k]},\,\,\,\,n=3,\cdots.
\end{eqnarray*}
Then we obtain that  $\g^{[\bullet]}$ is a descending
filtration of the vector space $\g$ and
\begin{eqnarray*}
[\g^{[m]},\g^{[n]}]\subset\g^{[m+n]},\,\,\,\g^{[m]}\rhd\g^{[n]}\subset\g^{[m+n]},\,\,\,\forall m,n\ge 0.
\end{eqnarray*}
Thus $(\g,\g^{[\bullet]})$ is a connected filtered post-Lie algebra.  Conversely, let  $(\g,\huaF_{\bullet}\g)$   be a connected  filtered post-Lie algebra. By $\huaF_{1}\g=\g$, we deduce that $\g^{[n]}\subset \huaF_{n}\g,~\,n\geq 1.$  
\begin{defi}
A post-Lie algebra $(\g,[\cdot,\cdot],\rhd)$ is called {\bf nilpotent} if there is $n\geq 1$ such that $\g^{[n]}=0$.
\end{defi}

\begin{rmk}
Let $(\g,\huaF_{\bullet}\g)$ be a connected complete  post-Lie algebra. By $\g=\huaF_1\g$ and (ii) in Proposition \ref{pronilpoten-post-lie}, we deduce that $\g/\huaF_{n}\g$ is a nilpotent post-Lie algebra. On the other hand, recall that a pronilpotent algebra is the inverse limit of nilpotent algebras. So by (iii) in Proposition \ref{pronilpoten-post-lie},  we find that a connected complete  post-Lie algebra $\g$ is a pronilpotent post-Lie algebra.
\end{rmk}

Let $(\g,[\cdot,\cdot],\rhd)$ be a nilpotent post-Lie algebra. Since there is $n\geq 1$ such that $\g^{[n]}=0$, it follows that $\g=\underleftarrow{\lim}~\g/\g^{[n]}$. Thus the following result is obvious.

\begin{pro}\label{completion-nilpotent-post-Lie}
Let $(\g,[\cdot,\cdot],\rhd)$ be a nilpotent post-Lie algebra. Then  $(\g,\g^{[\bullet]})$ is a connected complete post-Lie algebra.
\end{pro}

\begin{ex}
The completion of  $(Lie(\bk\{\huaO\}),\huaF_{\bullet}Lie(\bk\{\huaO\}))$ is a connected complete post-Lie algebra $(\widehat{Lie}(\bk\{\huaO\}),\huaF_{\bullet}\widehat{Lie}(\bk\{\huaO\}))$. More precisely, we have $\widehat{Lie}(\bk\{\huaO\})=\prod_{n=1}^{+\infty}\big(Lie(\bk\{\huaO\})\big)_n$ and
\[
\huaF_{n}\Big(\widehat{Lie}(\bk\{\huaO\})\Big):=\left\{
\begin{array}{ll}
\prod_{n=1}^{+\infty}\big(Lie(\bk\{\huaO\})\big)_n, &n=0, \\
\prod_{k=n}^{+\infty}\big(Lie(\bk\{\huaO\})\big)_k, &n\geq 1.
\end{array}
\right.
\]
\end{ex}

\begin{ex}
We consider the pre-Lie algebra structure on the vector space of power series ring $x^2\bk [[x]]$  given by
$
x^m\rhd x^n=nx^{m+n-1},$ for all $m,n\geq 2.$
Define \[
\huaF_n(x^2\bk [[x]]):=\left\{
\begin{array}{ll}
x^2\bk [[x]], &n=0, \\
x^{n+1}\bk [[x]], &n\geq 1.
\end{array}
\right.
\]
Then, $(x^2\bk [[x]],\huaF_{\bullet}x^2\bk [[x]])$ is a connected complete  pre-Lie algebra.
\end{ex}

\begin{ex}\label{finite-pre-Lie}
Let $p$ be an odd prime number. Similar to the discussion of the above example, $(x^2\mathbb F_p [[x]],\huaF_{\bullet}x^2\mathbb F_p[[x]])$ is a connected complete  pre-Lie algebra. Moreover, for $n\geq 1$, $\huaF_n(x^2\mathbb F_p [[x]])=x^{n+1}\mathbb F_p [[x]]$ is an ideal of the pre-Lie algebra $x^2\mathbb F_p [[x]]$. Thus, we conclude that the space
$$
x^2\mathbb F_p [[x]]/\huaF_n(x^2\mathbb F_p [[x]])\cong\oplus_{k=2}^{n}\mathbb F_px^k
$$
is a  nilpotent pre-Lie algebra of dimension $n-1$ over the finite  field $\mathbb F_p$.
\end{ex}

\begin{defi}
A {\bf complete post-Hopf algebra} $(H,\rhd,\huaF_{\bullet} H)$  is a post-Hopf  algebraic object in the category $(\cVec,\hat{\otimes},\bk)$. A complete Hopf algebra $(H,\rhd,\huaF_{\bullet} H)$ is called {\bf connected} if $H/\huaF_1 H=\bk.$
\end{defi}

\begin{ex}
Let $(\g,\huaF_{\bullet}\g)$ be a complete  post-Lie algebra. Then  $\Big(\big(U(\g),\cdot,\Delta^{\co},\rhd\big),\huaF_{\bullet}U(\g)\Big)$ is a filtered post-Hopf algebra. Since the completion functor $\widehat{(-)}$ is a  symmetric monoidal functor, we deduce that $\Big(\big(\widehat{U}(\g),\cdot,\Delta^{\co},\rhd\big),\huaF_{\bullet}\widehat{U}(\g)\Big)$ is a complete post-Hopf algebra.
\end{ex}

Let $(\g,\huaF_{\bullet}\g)$ be a connected complete  post-Lie algebra. Define the left multiplication $L^\rhd:\g\longrightarrow\gl(\g)$ by $L^\rhd_xy=x\rhd y,$ for all $x,y\in\g.$ By the completeness condition, $\exp(L^\rhd_x)=\sum_{n=0}^{+\infty}\frac{(L^\rhd_x)^n}{n!}\in GL(\g)$
is well-defined for all $x\in\g.$
Since $(\g,\huaF_{\bullet}\g)$ is connected, we deduce that $\Big(\big(\widehat{U}(\g),\cdot,\Delta^{\co},\rhd\big),\huaF_{\bullet}\widehat{U}(\g)\Big)$ is a connected complete post-Hopf algebra.  Forgetting the operation $\rhd$ gives a connected complete Hopf algebra $\Big(\big(\widehat{U}(\g),\cdot,\Delta^{\co}\big),\huaF_{\bullet}\widehat{U}(\g)\Big)$. By the Milnor-Moore
Theorem \cite[Theorem 7.3.26]{Fr} of connected complete Hopf algebras, the Lie algebra of primitive elements $P\big(\widehat{U}(\g),\cdot,\Delta^{\co}\big)$ is exactly $(\g,[\cdot,\cdot]_\g)$.  By \cite[Proposition 8.1.5]{Fr}, the exponential map
\begin{eqnarray}\mlabel{exp-log-1}
\exp:P\big(\widehat{U}(\g),\cdot,\Delta^{\co}\big)\lon G\big(\widehat{U}(\g),\cdot,\Delta^{\co}\big)
\end{eqnarray}  gives a canonical isomorphism from the set of primitive elements to the set of group-like elements of the connected complete Hopf algebra $\Big(\big(\widehat{U}(\g),\cdot,\Delta^{\co}\big),\huaF_{\bullet}\widehat{U}(\g)\Big)$.
Its inverse map is the logarithm function $\log$.

The connected filtered post-Hopf algebra
$\Big(\big(U(\g),\cdot,\Delta^{\co},\rhd\big),\huaF_{\bullet}U(\g)\Big)$
gives rise to the sub-adjacent connected filtered Hopf algebra
$\Big(\big(U(\g),*_\rhd,\Delta^{\co}\big),\huaF_{\bullet}U(\g)\Big)$
as well as the connected complete Hopf algebra $
\Big(\big(\widehat{U}(\g),*_\rhd,\Delta^{\co}\big),\huaF_{\bullet}\widehat{U}(\g)\Big).
$ Denote by $\exp_{*_\rhd}$ and $\log_{*_\rhd}$ the corresponding
exponential map and logarithm map.
 The {\bf post-Lie Magnus expansion} \mcite{CP,CEO,MQS}  is given by
\begin{eqnarray}\mlabel{post-Lie Magnus expansion}
 \Omega:\g\lon \g, \quad \Omega(x):=\log_{*_\rhd}(\exp(x)) \tforall x\in \g.
\end{eqnarray}

Let $(\g,\huaF_{\bullet}\g)$ be a connected complete  Lie algebra. Recall that the Baker-Campbell-Hausdorff formula of $(\g,\huaF_{\bullet}\g)$ is a group structure $\BCH:\g\times\g\lon \g$ on $\g$ which is given by
\begin{eqnarray*}
\BCH(x,y)= \log\big(\exp(x)\cdot\exp(y)\big)\tforall x,y\in\g.
\end{eqnarray*}

Utilizing the post-Lie Magnus expansion and the Baker-Campbell-Hausdorff formula, we obtain the formal integration   of  connected complete  post-Lie algebras.
\begin{thm}\mlabel{integration}
Let $(\g,\huaF_{\bullet}\g)$ be a connected complete  post-Lie algebra. Then there is a post-group structure on the set $\g$, which is given by
\begin{equation}
\mlabel{inte-1}x\cdot y=\BCH(x,y),\quad
x\rhd y=\exp(L^\rhd_{\Omega(x)})(y) \tforall x,y\in \g.
\end{equation}
The post-group $(\g,\cdot,\rhd)$ is called the {\bf formal integration} of the connected complete  post-Lie algebra $(\g,\huaF_{\bullet}\g)$.
\end{thm}
\begin{proof}
Since $\Big(\big(\widehat{U}(\g),\cdot,\Delta^{\co},\rhd\big),\huaF_{\bullet}\widehat{U}(\g)\Big)$ is a connected complete post-Hopf algebra, by Theorem \mref{group-like}, we deduce that $G\big(\widehat{U}(\g),\cdot,\Delta^{\co},\rhd\big)$ is a post-group. By \meqref{exp-log-1}, we obtain
$$
G\big(\widehat{U}(\g),\cdot,\Delta^{\co}\big)=\big\{\exp(x)\,\big| \tforall x\in \g\big\}.
\vspace{-.2cm}
$$
Thus, we have
\begin{eqnarray*}
\exp(x)\cdot\exp(y)=\exp\Big(\log\big(\exp(x)\cdot\exp(y)\big)\Big)=\exp\big(\BCH(x,y)\big) \tforall x,y\in\g.
\end{eqnarray*}
For all $x,y\in\g$, we have
\begin{eqnarray}\mlabel{Magnus expansion-to-inte}
\exp(L^\rhd_{\Omega(x)})(y)=\sum_{n=0}^{+\infty}\frac{(L^\rhd_{\Omega(x)})^n(y)}{n!}\stackrel{\meqref{u-3},\meqref{post-rbb-1}}{=}\exp_{*_\rhd}(\Omega(x))\rhd y\stackrel{\meqref{post-Lie Magnus expansion}}{=}\exp(x)\rhd y.
\end{eqnarray}
Moreover, we have
\begin{eqnarray*}
\exp(x)\rhd \exp(y)&=&\exp(x)\rhd \sum_{n=0}^{+\infty}\frac{y^n}{n!}=\sum_{n=0}^{+\infty}\frac{1}{n!}\exp(x)\rhd y^n\\
&\stackrel{\meqref{Post-hopf-2}}{=}&\sum_{n=0}^{+\infty}\frac{1}{n!}\underbrace{(\exp(x)\rhd y)\cdots(\exp(x)\rhd y)}_n\stackrel{\meqref{Magnus expansion-to-inte}}{=}\exp\Big(\exp(L^\rhd_{\Omega(x)})(y)\Big).
\end{eqnarray*}
By \meqref{exp-log-1}, transporting the post-group structure on $G\big(\widehat{U}(\g),\cdot,\Delta^{\co},\rhd\big)$ to that on the set $\g$, we obtain a post-group $(\g,\cdot,\rhd)$, where $\cdot$ and $\rhd$ are given by \meqref{inte-1}.
\end{proof}

\begin{rmk}
The  integration of nilpotent pre-Lie algebras was studied in
\mcite{Sm22a}. Since a nilpotent pre-Lie algebra is a connected complete pre-Lie algebra, the formal
integration of connected complete  post-Lie algebras extends the one given in \mcite{Sm22a}. Recently, there were  fruitful investigations on the relationship between finite braces and pre-Lie rings, see \cite{Sm22b,Sm22c,Sm22d} for more details. On the other hand, the
integration of post-Lie algebras was studied in \mcite{MQS,MQ} using $D$-bialgebras and post-Lie Magnus
expansions.
We use post-Hopf algebras and post-groups to study the integration of connected complete  post-Lie algebras.
\end{rmk}

The main tool used in our formal integration procedure \meqref{inte-1} is the  Baker-Campbell-Hausdorff
formula and the post-Lie Magnus expansion. On the other hand, it is known that the Baker-Campbell-Hausdorff formula also holds for nilpotent Lie algebras with orders $p^n$ and finite $p$-groups with orders $p^n$ provided $p > n + 1$. Applying this formula, various results on the relations between finite braces and pre-Lie algebras were obtained. Comparing the two approaches suggests the following question:

\smallskip

\noindent
{\bf Problem.} Find suitable constructions of ``differentiation'' and ``formal integration'' for finite post-groups and post-Lie algebras.

\smallskip

We end this subsection with providing a class of connected complete pre-Lie algebras from operads and giving their  formal integrations.

For an operad $\huaP$, a pre-Lie algebra structure \cite[Theorem 1.7.3]{KM} on $\prod_{n=1}^{+\infty}\huaP(n)_{\mathbb S_{n}}$ is given by
\vspace{-.3cm}
\begin{eqnarray}\mlabel{operad-to-pre-lie}
(\bar{a}\rhd\bar{b})_n=\sum_{k=1}^{n}\sum_{i=1}^{k}\overline{\gamma(b_{k};\Id,\ldots,a_{n+1-k},\ldots,\Id)}.
\end{eqnarray}
Consider the pre-Lie subalgebra $\widehat{\huaP}^+_{\mathbb S}\coloneqq \prod_{n=2}^{+\infty}\huaP(n)_{\mathbb S_{n}}$ and define a descending
filtration on $\widehat{\huaP}^+_{\mathbb S}$ by
\[
\huaF_k(\widehat{\huaP}^+_{\mathbb S}):=\left\{
\begin{array}{ll}
\prod_{n=2}^{+\infty}\huaP(n)_{\mathbb S_{n}}, &\mbox {$k=0$, }\\
\prod_{n=k+1}^{+\infty}\huaP(n)_{\mathbb S_{n}}, &\mbox {$k\ge 1$.}
\end{array}
\right.
\]

\begin{pro}\mlabel{operad-to-complete-pre-lie}
 With the above notations, $\big(\widehat{\huaP}^+_{\mathbb S},\huaF_{\bullet}(\widehat{\huaP}^+_{\mathbb S})\big)$ is a connected complete  pre-Lie algebra.
\end{pro}

\begin{proof}
By \meqref{operad-to-pre-lie}, for all $\bar{a}\in \huaF_{k}(\widehat{\huaP}^+_{\mathbb S}),~\bar{b}\in \huaF_{l}(\widehat{\huaP}^+_{\mathbb S})$,   we have $\bar{a}\rhd\bar{b}\in  \huaF_{k+l}(\widehat{\huaP}^+_{\mathbb S})$. Therefore,
\begin{eqnarray*}
\huaF_{k}(\widehat{\huaP}^+_{\mathbb S})\rhd \huaF_l(\widehat{\huaP}^+_{\mathbb S})\subset \huaF_{k+l}(\widehat{\huaP}^+_{\mathbb S}),\,\,\,\,k,l=1,2,\ldots,
\end{eqnarray*}
which implies that  $(\widehat{\huaP}^+_{\mathbb S},\huaF_{\bullet}(\widehat{\huaP}^+_{\mathbb S}))$ is a filtered pre-Lie algebra. Moreover, since $\huaP^+_{\mathbb S}=\huaF_1(\huaP^+_{\mathbb S})$ and $\widehat{\huaP}^+_{\mathbb S}\cong\underleftarrow{\lim}\widehat{\huaP}^+_{\mathbb S}/\huaF_n (\widehat{\huaP}^+_{\mathbb S})$, we deduce that $(\widehat{\huaP}^+_{\mathbb S},\huaF_{\bullet}(\widehat{\huaP}^+_{\mathbb S}))$ is a connected complete pre-Lie algebra.
\end{proof}

Theorem \mref{integration} and Proposition \mref{operad-to-complete-pre-lie} give a formal integration of the above pre-Lie algebra:

\begin{thm}\mlabel{O-to-pre-Lie-algebra-to-pre-group-cor}
Let $\huaP$ be an operad. Then there is a pre-group structure on the set $\widehat{\huaP}^+_{\mathbb S}$  given by
\vspace{-.1cm}
$$
x\cdot y=x+y,\quad
x\rhd y=\exp(L^\rhd_{\Omega(x)})(y) \tforall x,y\in \widehat{\huaP}^+_{\mathbb S}.
$$
\end{thm}

\subsection{Post-groups and Lie-Butcher groups}
We show that there is a post-group structure underlying the Lie-Butcher group. We also give an explicit description of the Lie-Butcher group structure using the above formal integration.

Since $\Big((Lie(\bk\{\huaO\}),[\cdot,\cdot],\rhd),\huaF_{\bullet}Lie(\bk\{\huaO\})\Big)$ is a connected filtered post-Lie algebra,   the universal enveloping algebra
$\Big((\bk\langle \huaO\rangle,\cdot,\Delta^{\co},\rhd),\huaF_{\bullet}\bk\langle \huaO\rangle\Big)$ is a connected filtered post-Hopf algebra. Its completion is the connected complete post-Hopf algebra $$\Big((\bk\langle\langle\huaO\rangle\rangle,\cdot,\Delta^{\co},\rhd),\huaF_{\bullet}\bk\langle\langle\huaO\rangle\rangle\Big).$$ Its elements are called {\bf Lie-Butcher series}.
The group of group-like elements of the sub-adjacent complete Hopf algebra $\Big((\bk\langle\langle\huaO\rangle\rangle,*_\rhd,\Delta^{\co}),\huaF_{\bullet}\bk\langle\langle\huaO\rangle\rangle\Big)$ given in Theorem \mref{thm:subHopf} is called the {\bf Lie-Butcher group} \cite{MF}, and denoted by $G_{\LB}$.

\begin{thm}\mlabel{thm:LB}
The set of the group-like elements $G(\bk\langle\langle\huaO\rangle\rangle,\cdot,\Delta^{\co},\rhd)$ of the complete post-Hopf algebra $(\bk\langle\langle\huaO\rangle\rangle,\cdot,\Delta^{\co},\rhd)$ is a post-group. Moreover,
the sub-adjacent  group  of the post-group $G(\bk\langle\langle\huaO\rangle\rangle,\cdot,\Delta^{\co},\rhd)$ is the  Lie-Butcher group $G_{\LB}$, giving  the commutative diagram
  $$\small{
 \xymatrix{\text{complete post-Hopf alg.}~(\bk\langle\langle\huaO\rangle\rangle,\cdot,\Delta^{\co},\rhd)  \ar[rr]^{\qquad sub\mbox{-}adjacent}\ar[d]_{\text{group-like}}^{\text{elements}} &  & \text{complete Hopf alg.}~ (\bk\langle\langle\huaO\rangle\rangle,*_\rhd,\Delta^{\co}) \ar[d]_{\text{group-like}}^{\text{elements}}  \\
\text{post-group}~G(\bk\langle\langle\huaO\rangle\rangle,\cdot,\Delta^{\co},\rhd)  \ar[rr]^{\qquad sub\mbox{-}adjacent} & &\mbox{Lie-Butcher group}~ G_{\LB}.}
}
 $$
\end{thm}
\begin{proof}
  It follows from  Theorem \mref{group-like} directly.
\end{proof}

By Theorem \mref{integration} and the fact that $P(\bk\langle\langle\huaO\rangle\rangle,\cdot,\Delta^{\co})=\widehat{Lie}(\bk\{\huaO\})$, we have the following concrete characterization of the post-group $G(\bk\langle\langle\huaO\rangle\rangle,\cdot,\Delta^{\co},\rhd)$ and Lie-Butcher group $G_{\LB}$.
\begin{thm}
The set of group-like elements $G(\bk\langle\langle\huaO\rangle\rangle)$ is equal to $\{\exp(x)\,|\,  x\in \widehat{Lie}(\bk\{\huaO\})\}$ and the post-group structure is given by
\begin{eqnarray*}
\exp(x)\cdot \exp(y)&=&\exp\big(\BCH(x,y)\big),\\
\exp(x)\rhd \exp(y)&=&\exp\Big(\exp(L^\rhd_{\Omega(x)})(y)\Big) \tforall x,y\in \widehat{Lie}(\bk\{\huaO\}).
\end{eqnarray*}
Moreover, the Lie-Butcher group $G_{\LB}=\big\{\exp(x)\,\big| x\in \widehat{Lie}(\bk\{\huaO\})\big\}$ is given by
\begin{eqnarray*}
\exp(x)*_\rhd \exp(y)=\exp\Big(\BCH\big(x,\exp(L^\rhd_{\Omega(x)})(y)\big)\Big) \tforall x,y\in \widehat{Lie}(\bk\{\huaO\}).
\end{eqnarray*}
\end{thm}

\noindent
{\bf Acknowledgments.}  This work is supported by
NSFC (11931009,11922110,12001228,12271265), Sino-Russian Mathematical Center and the Fundamental Research Funds for the Central Universities and Nankai Zhide Foundation. We give our warmest thanks to the referee for the very helpful suggestions that improve the paper.

\noindent
{\bf Declaration of interests. } The authors have no conflicts of interest to disclose.

\noindent
{\bf Data availability. } Data sharing is not applicable to this article as no new data were created or analyzed in this study.

\end{document}